\documentclass[12pt]{article}
\usepackage{amsmath, amsthm, amssymb}

\theoremstyle{plain}
\newtheorem{theorem}{Theorem}
\newtheorem{lemma}[theorem]{Lemma}

\newtheorem{fact}[theorem]{Fact}
\newtheorem{corollary}[theorem]{Corollary}
\newtheorem*{theorem*}{Theorem}
\newtheorem*{lemma*}{Lemma}

\theoremstyle{definition}
\newtheorem{definition}[theorem]{Definition}
\newtheorem*{definition*}{Definition}

\newtheorem*{example*}{Example}
\theoremstyle{remark}
\newtheorem{remark}[theorem]{Remark}
\newtheorem*{remark*}{Remark}

\newcommand{\comment}[1]{}
\newcommand{\extracomputation}[1]{}

\newcommand{\FamilyFont}[1]{\textsc{#1}}
\newcommand{\Conn}{\ensuremath{\FamilyFont{Conn}}}
\newcommand{\Ham}{\ensuremath{\FamilyFont{Ham}}}
\newcommand{\Chr}[1]{\ensuremath{\FamilyFont{Ch}_{#1}}}
\newcommand{\Pln}{\ensuremath{\FamilyFont{Pln}}}

\newcommand{\DistributionFont}[1]{\mathrm{#1}}
\newcommand{\Bin}{\DistributionFont{Bin}}
\newcommand{\Pois}{\DistributionFont{Pois}}
\makeatletter
\def\imod#1{\allowbreak\mkern5.5mu({\operator@font mod}\,\,#1)}
\makeatother

\newcommand{\predfont}[1]{\ensuremath{\mathrm{#1}}}

\newcommand{\langfont}[1]{\ensuremath{\mathcal{#1}}}
\newcommand{\Logic}{\ensuremath{\langfont{L}}}
\newcommand{\FO}{\ensuremath{\langfont{FO}}}
\newcommand{\SO}{\ensuremath{\langfont{SO}}}
\newcommand{\MSO}{\ensuremath{\langfont{MSO}}}
\newcommand{\MESO}{\ensuremath{\langfont{MESO}}}

\newcommand{\E}{\mathbb{E}} 
\newcommand{\VAR}{\mathrm{Var}} 
\newcommand{\COV}{\mathrm{Cov}} 
\newcommand{\logstar}{\log^*}

\newcommand{\codeg}{\ensuremath{\mathrm{codeg}}}

\newcommand{\defeq}{\ensuremath{:=}}

\author{Simi Haber
\and Saharon Shelah
}
\title{Random graphs and Lindstr\"{o}m quantifiers for natural graph properties}
\date{}
\begin{document}
\maketitle

\begin{abstract}
We study zero-one laws for random graphs. We focus on the following
question that was asked by many: Given a graph property $P$, is
there a language of graphs able to express $P$ while obeying the
zero-one law? Our results show that on the one hand there is a
(regular) language able to express connectivity and $k$-colorability
for any constant $k$ and still obey the zero-one law. On the other
hand we show that in any (semiregular) language strong enough to
express Hamiltonicity one can interpret arithmetic and thus the
zero-one law fails miserably. This answers a question of Blass and
Harary.
\end{abstract}

\section{Introduction}

\subsection{Motivation and definitions}
In this paper we study questions related to the following
observation: Natural graph properties are either held by most graphs
or not held by most graphs.

The \emph{binomial random graph} $G(n,p)$ is a probability
distribution over all labeled graphs of order $n$. For convenience
we assume that the vertex set is simply the set $\{1,2,\dotsc,n\}$.
Sampling from $G(n,p)$ is done by including every possible edge in
the graph at random with probability $p$ independently of all other
edges. Random graphs were extensively studied from the sixties,
starting with the seminal work of Erd\H{o}s and R\'{e}nyi
\cite{A:erdos_&_renyi1959, IP:erdos_&_renyi1960}. From the beginning
of the study of random graphs a phenomenon was spotted. Consider the
simplest model of random graph, $G(n,1/2)$, in which we pick a graph
uniformly at random from all the labeled graphs of $n$ vertices. For
many natural graphs properties --- properties like connectivity,
containing a Hamilton path or cycle, not being colorable in some
fixed number of colors, not being planar and containing a small
fixed graph as a subgraph
--- $G(n,1/2)$ \emph{asymptotically almost surely} (abbreviated
a.a.s.) has these properties. This means that the probability for
which a random graph has any one of these properties tends to one as
the size of the graph tends to infinity. We may remove the negation
before colorability and planarity and say that for any of these
properties either a.a.s.\ $G(n,1/2)$ has these property or a.a.s.\
$G(n,1/2)$ does not have it.

The phenomenon remains valid if we replace $1/2$ by $p$ for any
constant $0 < p < 1$, so from now on $p$ will be some fixed constant
real number (strictly between zero and one) representing the edge
probability of the random graph.

This phenomenon suggests the following definition:
\begin{definition}
Let $\mathcal{A}$ be a set of graph properties. We say that
\emph{$\mathcal{A}$ obeys the zero-one law} if for every property
$\varphi \in \mathcal{A}$ one has
\[ \lim_{n\rightarrow\infty} \Pr [G(n,p) \text{ has } \varphi] \in \{0,1\} .\]
\end{definition}
Having this definition we may rephrase the first sentence of this
introduction by saying that if $\mathcal{A}$ is a set of natural
graph properties then $\mathcal{A}$ obeys the zero-one law. But what
are ``natural graph properties''? %
\comment{We will not try to answer
this question --- we simply study graph properties that are studied
by graph theorists. That said, we can still say what we consider a
\emph{set} of natural graph properties is (actually, what is not
such a set). If we agree that connectivity is a natural graph
property, so being disconnected seems natural as well. Similarly, if
being bipartite is a natural graph property, and being 3-colorable
is a natural property, then being either bipartite or 3-colorable
should also be considered a natural property. That is, we want
$\mathcal{A}$ to be closed under negation and disjunction. This
leads us to study properties that are expressible in a \emph{formal
language}.}%
A natural interpretation is: all properties defined by a sentence in
a language $\Logic$. Indeed, the first zero-one law is of this
sort\footnote{In fact, to the best of our knowledge, in \emph{all}
of the zero-one laws for graphs the set of properties is the set of
sentences in some formal language.} where the language was \FO, the
first order language (Glebski\u{\i} et al.\ in
\cite{A:glebskii_&_kogan_&_liogonkii_&_talanov69} and Fagin in
\cite{A:fagin76}, see below). Unfortunately, most classical graph
theoretic properties are not expressible in \FO. A reasonable list
of the most extensively studied graph theoretic properties should
include connectivity, $k$-colorability, Hamiltonicity and also
planarity. Thus we suggest to consider the language for graphs
obtainable from first order logic strengthen by a property from this
list. This is done by adding a quantifier for the property, as
detailed below.

Considering graphs, the most basic formal language is the first
order language of graphs denoted here by \FO. It consists of the
following symbols:
\newcommand{\adj}{\sim}
\begin{enumerate}
\item Variables, denoted along this paper by lower case Latin letters
$x,y,z$. In \FO, variables stand for vertices solely.
\item Relations. There are exactly two of these: adjacency, denoted here
by $\adj$, and equality, denoted as usual by =. Thus it is possible
to write $x=y$ or $y\adj z$. %
The relations is what makes this language a language of graphs.
\item Quantifiers. Again, there are two, the existential $\exists$
and the universal $\forall$. These can be applied only on
variables which means that quantification in \FO\ is only possible over
vertices. %
\item Boolean connectives, like $\wedge, \vee, \neg$ and
$\rightarrow$.
\end{enumerate}
Notice that there are no constants and no functions in this
language. As usual, we shall also use parentheses and punctuation
marks for the benefit of readability.

For example, in \FO\ we may write
\[ \exists x \exists y \exists z.\, [\neg(x=y) \wedge \neg(x=z) \wedge
\neg(y=z) \wedge x\adj y \wedge x\adj z \wedge y\adj z] , \] %
which means ``there exists a triangle in the graph''. Another
example might be ``there are no isolated vertices'':
\[ \forall x \exists y.\, [\neg(x=y) \wedge x\adj y] . \]

There are few limitations on the language. Every formula must be of
finite length, and again, variables stand only for vertices so in
particular we may only quantify over vertices. This is a crucial
difference between \mbox{first order} and higher order logics.

In the \emph{second order} language of graphs, denoted here by \SO,
we have \emph{relational variables} (also called second order
variables). Adhering customary notation, we denote relational
variables by capital Latin letters. We write the arity of a second
order variable in superscript near the binding operator. If $A$ is
an unary relation, we use $x \in A$ instead of $A(x)$, and if $B$ is
a binary relation we may write $xBy$ instead of $B(x,y)$. Here are a
few examples of sentences in \SO\ representing graph properties:
\begin{description}
\item[Connectivity] A graph is connected if there is a
walk or a path\footnote{A \emph{walk} in a graph $G=(V,E)$ is a
sequence of vertices $v_0,v_1,\dotsc,v_l$ such that for any $1 \leq
i \leq l$ one has $v_{i-1} \adj v_i$. If additionally the vertices
are all distinct, it is called a \emph{path}. A \emph{cycle} is a
walk in which the first and last vertices are identical, and all
others are distinct.} between any two vertices. Equivalently, there
is an edge between the parts of any nontrivial partition:
\[ \forall A^1.\, [((\exists x.\, x \in A) \wedge (\exists x.\, x \notin A)) \rightarrow
(\exists x \in A, y\notin A.\, x \adj y) ] . \]
\item[3-colorability] %
A graph is $k$-colorable if there is a partition of the vertex set
into $k$ parts (colors), such that there is no edge between ny
vertices in the same part. This may be written as:
\begin{align*}
\exists R^1,G^1,B^1.\, [ &\predfont{Partition}(R,G,B)    \wedge \\ %
&\wedge \left( \forall x,y.\, (x,y\in R)\rightarrow x \not\sim y \right) \wedge \\ %
&\wedge \left( \forall x,y.\, (x,y\in G)\rightarrow x \not\sim y \right) \wedge \\ %
&\wedge \left( \forall x,y.\, (x,y\in B)\rightarrow x \not\sim y \right) ] , %
\end{align*} %
where $\predfont{Partition}(R,G,B)$ is a first order formula saying
that the sets $R,G$ and $B$ are mutually disjoint and their union
equals the set of vertices.
\item[Hamiltonicity] %
A Hamilton path in a graph is a path that visits every vertex
exactly once. A Hamilton cycle is a Hamilton path where the first and last vertices are adjacent. We say that a graph is Hamiltonian if it contains a Hamilton cycle.

In the following sentence $\predfont{Min}(\mathord{<},x)$ is a first
order formula expressing the fact that $x$ is a minimal element with
respect to an order $\mathord{<}$ (and similarly for
$\predfont{Max}(\mathord{<},y)$). The first three lines say that
$\mathord{<}$ is a linear order, the fourth line says that each
vertex is adjacent to its successor in that order and the last line
adds that the first vertex is adjacent to the last one.
\begin{align*}
\exists \mathord{<}^2.\, [ & \left( \forall x.\, \neg(x<x) \right) \wedge \\
               & \left( \forall x,y.\, ((x<y) \vee (y<x)) \wedge \neg
                 ((x<y) \wedge (y<x)) \right) \wedge \\ %
               & \left( \forall x,y,z.\, ((x<y) \wedge (y<z))
                 \rightarrow (x<z) \right ) \wedge \\ %
               & \left( \forall x,y.\, (\neg \exists z. (x<z) \wedge
                 (z<y)) \rightarrow x \sim y \right) \wedge \\ %
               & \exists x,y.\, \predfont{Min}(\mathord{<},x) \wedge
                 \predfont{Max}(\mathord{<},y) \wedge y \adj x] . %
\end{align*} %
\end{description}

\subsection{Previous results}
The first zero-one law was proven by Glebski\u{\i} et al.\ in
\cite{A:glebskii_&_kogan_&_liogonkii_&_talanov69} and independently
by Fagin \cite{A:fagin76}. They showed that the set of properties
describable in \FO\ obeys the zero-one law. In his paper Fagin
demonstrated that \SO\ does not obey the zero-law by expressing the
sentence ``the number of vertices is even'':
\begin{equation} \label{eq: Fagin's parity}
\exists P^2.\, \left[\forall x.\, \exists! y.\, (x \neq y \wedge xPy
\wedge yPx ) \right] .
\end{equation}
The quantifier $\exists!$ stands for there exists exactly one
element such that... Notice that parity is a property of a set (of
the set of vertices in our case) and thus adjacency does not appear
in this sentence. When the zero-one law fails for $\mathcal{A}$, we
may argue for less.
\begin{definition}
We say that
$\mathcal{A}$ obeys a \emph{limit law} if for every property $p \in
\mathcal{A}$ the limit
\[ \lim_{n\rightarrow\infty} \Pr[G(n,p) \text{ has } P] \]
exists.
\end{definition}
Clearly, the sentence in Equation \ref{eq: Fagin's parity}
demonstrates that \SO\ does not even obey the limit law.

At this point it is worth mentioning that \FO\ is rather weak. Of
the list of ``natural'' properties above --- connectivity,
Hamiltonicity, $k$-colorability (for $k\geq 2$), planarity and
having a fixed graph as a subgraph --- only the last is a first
order property. Thus it is clear that the zero-one law of Fagin and
Glebski\u{\i} et al.\ does not capture the phenomenon
aforementioned. On the other hand, as Fagin showed, \SO\ is too
expressive. There are many papers dealing with the problem of
finding a language that is strong enough to express some graph
properties that are not first order expressible on the one hand,
while still obeying the zero-one law on the other. In the following
we will mention some of these results. The following description is
far from being comprehensive. For a survey of the results in this
field see, e.g., \cite{IC:compton89}.

Looking at examples such as above (in particular the sentences
representing connectivity and $k$-colorability), it seems reasonable
to study the expressive power of the \emph{monadic second order}
logic. In the monadic second order language of graphs, denoted here
by \MSO, all second order variables must be of arity one, that is,
all second order variables represent sets. It was asked by Blass and
Harary in 1979 \cite{A:blass_&_harary1979} whether \MSO\ obeys the
zero-one law. In their 1985 paper \cite{A:Kaufmann_&_Shelah1985},
Kaufmann and Shelah provided a strong negative answer --- even a
fraction of \MSO\ is enough to express properties for which the
asymptotic probability does not exist. Let \MESO\ be the set of
formulas having the following structure: $\exists \bar{A}.\,
\varphi(\bar{A}, \bar{a})$, where $\bar{A}$ is a vector of unary
second order variables, $\bar{x}$ are first order variables and
$\varphi$ is a first order formula. That is, in \MESO\ we may only
existentially quantify over sets and at the beginning of the
formula. The 3-colorability sentence in the second example above is
a \MESO\ sentence. Notice that \MESO\ is not closed under negation.
In particular, connectivity is not \MESO-expressible, while being
disconnected obviously is. Kaufmann and Shelah showed that one can
interpret a segment of arithmetic in \MESO, and hence express
properties like, say, $0 \leq \sqrt{n} \leq 4 \imod{10}$. Clearly
the last property has no asymptotic probability; moreover, the limit
superior of its probability sequence is one and the limit inferior
is zero.\comment{A formal definition of arithmetization is given
before the statement of the new results of this study (see
Definition \ref{def: arithmetization} in Section \ref{subsec:
results} (page \pageref{def: arithmetization})).}

In order to give the full strength formulation of the result of
Kaufmann and Shelah let us define the notion of
\emph{arithmetization}. The language of arithmetic is the first
order language with universe set $\mathbb{N}$ and vocabulary
$\{<,+,\cdot\}$ where $<$ is the natural order of integers and
$+,\cdot$ have their usual meaning (see, e.g.,
\cite{B:ebbinghaus_&_flum_&_thomas1984}).
\begin{definition} \label{def: arithmetization}
Let $\Logic$ be a language for the class of graphs and let
$\mathcal{G} = (G_n)$ be a sequence of probability distributions
over graphs of order $n$. We say that \emph{the pair $(\Logic,
\mathcal{G})$ can interpret arithmetic for a function $f\colon
\mathbb{N} \to \mathbb{N}$} if for every sentence $\varphi$ in the
language of arithmetic there is a sentence $\psi_\varphi \in \Logic$
such that
\[ \lim_{n \rightarrow \infty} \Pr \left[ G_n \models \psi_\varphi \Longleftrightarrow  \mathbb{N} |_{f(n)} \models \varphi \right] = 1 . \]
If the pair $(\Logic, \mathcal{G})$ can interpret arithmetic for
some function $f\colon \mathbb{N} \to \mathbb{N}$ we say that
\emph{the pair $(\Logic, \mathcal{G})$ has arithmetization}.
\end{definition}

When a pair $(\Logic, \mathcal{G})$ has arithmetization it is, in a
sense, the farthest that can be from obeying a zero-one law. For
example, in this situation for any recursive real $\alpha \in [0,1]$
there is a sentence $\varphi_\alpha \in \Logic$ having asymptotic
probability $\alpha$. The aforementioned result of Kaufmann and
Shelah about $\MESO$ is of this sort, demonstrating that $(\MESO,
G(n,1/2))$ has arithmetization. There are situations in which $\FO$
has arithmetization. Let $\alpha \in (0,1)$ be rational, then the
pairs $(\FO, G(n,p=n^\alpha))$ and $(\FO, G_{n,d=n^{1-\alpha}})$
have arithmetization (where $G_{n,d=n^{1-\alpha}}$ is the
\emph{random regular graph} with degree $d$. See
\cite{A:shelah_&_spencer88} and \cite{A:haber_&_krivelevich2010}
respectively). Also, in \cite{A:haber_&_mullerpp} it is showed that
if $d \geq 2$ is a constant integer and $r$ is a small enough
constant then the pair $(\FO,G(n;\mathbb{T}^d,r))$ has
arithmetization (where $\mathbb{T}^d$ is the $d$-dimensional torus
and $G(n;\mathbb{T}^d,r)$ is the \emph{random geometric graph} with
distance parameter $r$).

In view of the last result it seems reasonable to look for less
expressive fragments of second order logic for which the zero-one
law will hold. In their 1987 paper \cite{IP:kolaitis_&_vardi1987}
Kolaitis and Vardi proved a zero-one law for the \emph{\mbox{strict
$\Sigma^1_1$}} language --- the set of sentences of the form
$\exists \bar{S}.\, \psi(\bar{S})$ where $\psi$ is from the
\emph{Bernays-Sch\"{o}nfinkel} class. That is, the set of all
sentences of the form $\exists \bar{S}.\, \exists \bar{x}.\, \forall
\bar{y}.\, \varphi(\bar{S},\bar{x},\bar{y})$ where $\bar{S}$ is a
vector of second order variables, $\bar{x}$ and $\bar{y}$ are
vectors of first order variables and
$\varphi(\bar{S},\bar{x},\bar{y})$ is a quantifier free formula.
Notice that 3-colorability is a \mbox{strict $\Sigma^1_1$} property,
as well as disconnectivity. On the other hand connectivity is not a
\mbox{strict $\Sigma^1_1$} property. A line of research was started
by \cite{IP:kolaitis_&_vardi1987}, aiming to characterize the
$\Sigma_1^1$ fragments defined by first-order prefix classes
according to adherence to the zero-one law. The classification was
completed in \cite{IP:lebars1998}. See
\cite{IP:kolaitis_&_vardi2000} for a survey.

Another family of languages studied in this context is the family of
languages one get from adding a recursive operator to the first
order logic. These languages can express connectivity but not
$k$-colorability. There are a few such languages known to obey the
zero-one law \cite{A:talanov1981, A:blass_&_gurevich_&_Kozen1985}.

The last result we shall mention in this section deals with a
different strengthening of the first order logic. In
\cite{IP:kolaitis_&_kopparty2009} Kolaitis and Kopparty considered
the language one gets by augmenting a \emph{parity} quantifier to
\FO. Their result was a \emph{modular} limit-law:
\begin{theorem}[{\cite{IP:kolaitis_&_kopparty2009}}] \label{thm: kolaitis and kopparty}
Let $\FO[\oplus]$ be the regular language obtained by adding parity to the
first order language. Then for every property $P \in \FO[\oplus]$ there are
two rational numbers $\alpha_0, \alpha_1$ such that
\[ \lim_{n=2k+i \rightarrow \infty} \Pr[G(n,p) \text{ has } P] = \alpha_i \]
\end{theorem}
The result generalizes in the natural way to general modulo $k$
operators for any constant $k$. Theorem \ref{thm: kolaitis and
kopparty} means that the fact that a language is able to express
parity is actually not that bad. Indeed, parity has no asymptotic
probability, but we can think of the graph
sequence\footnote{Actually, we have a sequence of probability spaces
over graphs.} as two separate sequences, odd and even, and then we
do get a limit for the asymptotic probability in each of these
sequences. Adding parity clearly lets us say different things for
odd and even graphs, Theorem \ref{thm: kolaitis and kopparty} tells
us that it does not give more in terms of expressive power.

As mentioned above, there are many other papers along this line. To
the best of our knowledge none of these papers presented a formal
language strong enough to express $3$-colorability and connectivity
while obeying the zero-one law, The same question for Hamiltonicity
was explicitly asked first in \cite{A:blass_&_harary1979} and then
by many others.

The results in this paper (Theorems \ref{thm: conn+chr obey 0-1} and
\ref{thm: L(Q_ham) has arithmetic}) give the following answer to the
question above: On the one hand, there is a regular language able to
express connectivity and $k$-colorability for any fixed $k\geq 2$
while obeying the zero-one law. On the other hand, our main result
states that any semiregular language able to express Hamiltonicity
can express arithmetic as well (and therefore it violates even the
modular limit law of Kolaitis and Kopparty). Planarity behaves
similarly to connectivity and colorability. A result including
planarity will be published elsewhere.

Before stating our theorems we shall present the connected notions
of regular languages and Lindstr\"{o}m quantifiers.

\subsection{Regular languages and Lindstr\"{o}m quantifiers}
There is a trivial ``language'' that can express the properties
listed above while simultaneously obeying the zero-one law ---
simply take $\FO$ and add the sentences ``$G$ is Hamiltonian'',
``$G$ is planar'' and so on. Clearly this language misses some
notion of closure. To avoid such trivialities we need to define what
kind of languages are accepted. The definitions in this section
follow the ideas of Lindstr\"{o}m \cite{A:lindstrom1966c,
A:lindstrom1969}. Our notation is taken from
\cite{B:ebbinghaus_&_flum2006}, in which a full treatment of the
notions in this section may be found.

A language $\Logic$ is called \emph{semiregular} if it is closed
under ``first order operations'' and it is also closed under
substitution of a formula for a predicate. That is, $\Logic$ is
required to contain the atomic formulas (in our case, formulas of
the form $x=y$ and $x\adj y$), to be closed under Boolean
connectives (e.g., $\neg$ and $\wedge$) and existential
quantification and finally to allow redefinition of the predicates
through formulas in $\Logic$. A language is called \emph{regular} if
it is semiregular and closed under relativization
--- the operation of replacing the universe by an $\Logic$ definable set. Let us
give a concrete definition for the case of graphs:
\begin{definition} {\ }
\begin{enumerate}
\item A language of graphs $\Logic$ is said to be \emph{semiregular} if:
\begin{itemize}
\item All atomic formulas are in $\Logic$.
\item If $\varphi \in \Logic$ then $\neg \varphi \in \Logic$.
\item If $\varphi,\psi \in \Logic$ then $\varphi \wedge \psi \in
 \Logic$.
\item If $\varphi(x) \in \Logic$ then $\exists x.\, \varphi(x) \in
 \Logic$.
\item (Weak substitution) If $\varphi(x,y),\psi \in \Logic$ and $\varphi$ is
 anti-reflexive and symmetric, then there exists a sentence $\psi'
 \in \Logic$ such that
 \[ G \models \psi \Longleftrightarrow (V, \{ (x,y) \in V\times V
 \mid G\models  \varphi(x,y)\} ) \models \psi' \]
 where $G = (V,E)$ is a graph.
\end{itemize}
\item $\Logic$ is said to be \emph{regular} if in addition
\begin{itemize}
\item (Full substitution) Let $\varphi(x)$ and $\psi$ be formulas in $\Logic$ and
denote $V_\varphi = \{ x \in V \mid G\models \varphi(x)\}$. Then
there exists a sentence
 $\psi' \in \Logic$ such that
 \[ G \models \psi \Longleftrightarrow (V_\varphi, G[V_\varphi]) \models \psi' \]
 where $G = (V,E)$ is a graph, and $G[V']$ is the graph spanned on
 the vertex subset $V' \subset V$.
\end{itemize}
\end{enumerate}
\end{definition}

In order to get the minimal semiregular language that can express a
property $K$ we use generalized quantifiers or \emph{Lindstr\"{o}m
quantifiers}. Let $K$ be a property of graphs. We think of $K$ as
the set of all graphs having this property. As a graph property, $K$
is closed under isomorphism. Given $K$ we define the graph language
$\Logic(Q_K)$ as follows.
\begin{definition} {\label{def: Lindstrom Q, Flum's variant}\ }
\begin{enumerate}
\item
Let $K$ be a graph property. The set of formulas $\Logic(Q_K)$ is
the closure of the atomic formulas by the conjunction and negation
connectives, the existential quantifier and another quantifier
$Q_K$. The syntax of the new quantifier is $Q_K x y.
\varphi(x,y,\bar{a})$ where $\varphi(x,y,\bar{a}) \in \Logic(Q_K)$
is an antireflexive and symmetric formula in which $x,y$ are free
variables and $\bar{a}$ are parameters.

\item Given a graph $G=(V,E)$, the satisfaction of formulas of the
form $Q_K xy.\ \varphi(x,y,\bar{a})$ is determined by
\[ G \models Q_K xy.\, \varphi(x,y,\bar{a}) \Longleftrightarrow
(V,\{(x,y) \in V\times V \mid G \models \varphi(x,y,\bar{a})\}) \in
K .
\]
\end{enumerate}
\end{definition}
Of course, Lindstr\"{o}m quantifiers are not restricted to graphs
and can be defined for any vocabulary (indeed, usually that is the
case).

Here are three simple examples over sets: First notice that
\[ \forall x.\, \varphi(x) \Longleftrightarrow Q_{\{U\}} x.\, \varphi(x) \]
where $U$ is the universe set. The expression on the right hand side
means that we consider the set of all $x$'s for which $\varphi(x)$
holds, and then we check if this set belongs to the singleton
$\{U\}$, that is, if it is the whole universe. Similarly $\exists
x.\, \varphi(x) \Longleftrightarrow Q_{P(U) \setminus
\{\varnothing\}} x.\, \varphi(x)$ where $P(U)$ is the power set of
$U$. As a final example we mention that the parity quantifier of
\cite{IP:kolaitis_&_kopparty2009} mentioned above may be expressed
as a Lindstr\"{o}m quantifier e.\,g.\ by writing  $\oplus x.\,
\varphi(x) \Longleftrightarrow Q_{\{ A \in P(U) \mid |A| \text{ is
even}\}} x.\, \varphi(x)$.

The basic result regarding the Lindstr\"{o}m quantifiers is that for any property $K$, the language
$\Logic(Q_K)$ is the smallest semiregular language that can express
$K$ \cite{A:lindstrom1966c}.

The Lindstr\"{o}m quantifier $Q_K$ acts as an oracle that let us
check if the graph that we get by replacing the edge relation with a
defined one is in $K$. We can create more expressive languages by
allowing more freedom in defining the graphs to be queried. We
describe three variants here, in increasing expressive power.

Given a graph property $K$ we define the \emph{relativized} language
$\Logic^{\mathrm{rl}}(Q_k)$ similarly to the above, but this time we
have another formula defining which of the original vertices are
included in the queried graph. We denote the resulting language by
$\Logic^{\mathrm{rl}}(Q_k)$.
\begin{definition} {\ }
\begin{enumerate}
\item
Let $K$ be a graph property. The set of formulas
$\Logic^{\mathrm{rl}}(Q_k)$ is the closure of the atomic formulas by
first order operations together with additional quantifier $Q_K$.
The syntax of the new quantifier is given by $Q_K vxy.\,
\psi(v,\bar{a}), \varphi(x,y,\bar{a})$ where $\varphi(x,y,\bar{a})
\in \Logic^{\mathrm{rl}}(Q_k)$ is an antireflexive and symmetric
formula in which $x,y$ are free variables, $\psi(v,\bar{a}) \in
\Logic^{\mathrm{rl}}(Q_k)$ is a formula in which $v$ is free and
$\bar{a}$ are parameters.

\item Given a graph $G=(V,E)$, the satisfaction of formulas of the
form $Q_K vxy.\, \psi(v,\bar{a}), \varphi(x,y,\bar{a})$ is
determined as follows. Let $V|_\psi = \{v\in V \mid
\psi(v,\bar{a})\}$. Then
\begin{align*}
&  G \models Q_K vxy.\, \psi(v,\bar{a}), \varphi(x,y,\bar{a}) \Longleftrightarrow \\
&(V|_\psi,\{(x,y) \in V|_\psi \times V|_\psi \mid
\varphi(x,y,\bar{a})\}) \in K .
\end{align*}
\end{enumerate}
\end{definition}
It is easy to verify that $\Logic^{\mathrm{rl}}(Q_k)$ is regular for
any $K$, Moreover, it is the inclusion minimal regular language that
is able to express $K$.

In the next variant we allow to redefine the other predicate in the
graph vocabulary as well, namely, the equality predicate. This means
that the vertex set will be the quotient set of some equivalence
relation. We denote the resulting language by
$\Logic^{\mathrm{eq}}(Q_k)$.
\begin{definition} {\ }
\begin{enumerate}
\item
Let $K$ be a graph property. The set of formulas
$\Logic^{\mathrm{eq}}(Q_k)$ is the closure of the atomic formulas by
first order operations together with additional quantifier $Q_K$.
The syntax of the new quantifier is given by $Q_K vuwxy.
\psi(v,\bar{a}), \varphi_=(u,w,\bar{a}), \varphi_\adj(x,y,\bar{a})$
where $\varphi_\adj(x,y,\bar{a}) \in \Logic^{\mathrm{eq}}(Q_k)$ is
an antireflexive and symmetric formula in which $x,y$ are free
variables, $\varphi_=(u,w,\bar{a}) \in \Logic^{\mathrm{eq}}(Q_k)$ is
a reflexive, symmetric and transitive formula in which $u,w$ are
free variables, $\psi(v,\bar{a}) \in \Logic^{\mathrm{eq}}(Q_k)$ is a
formula in which $v$ is a free variable and $\bar{a}$ are
parameters.

\item Given a graph $G=(V,E)$, the satisfaction of formulas of the
form $Q_K vuwxy. \psi(v,\bar{a}), \varphi_=(u,w,\bar{a}),
\varphi_\adj(x,y,\bar{a})$ is determined as follows. Let $V' = \{v
\in V \mid \psi(v,\bar{a})\}$ and let $R$ be the equivalence
relation induced over $V'$ by $\varphi_=$ (that is, $R = \{(u,w) \in
V' \times V' \mid \psi_=(u, w, \bar{a}) \}$). Then the semantics of
$Q_K$ in $\Logic^{\mathrm{eq}}(Q_K)$ is given by
\begin{align*}
&  G \models Q_K vuwxy. \psi(v,\bar{a}), \varphi_=(u,w,\bar{a}),
\varphi_\adj(x,y,\bar{a}) \Longleftrightarrow \\ %
&(V' / R , \{([x],[y]) \in (V' / R) \times (V' / R) \mid
\varphi_\adj(x,y,\bar{a})\}) \in K .
\end{align*}
\end{enumerate}
\end{definition}
Generally, care must be taken in order to make sure that the edges
induced by $\psi_\adj$ are well defined.

The strongest variant lets us redefine the vertices as \emph{tuples}
of vertices. Let $l$ be an integer.  The syntax and the semantics of
$\Logic^{\mathrm{tu}}(Q_k)$ are simply vectorized versions of the
syntax and semantics of $\Logic^{\mathrm{eq}}Q_k)$. The formula
$\psi(\bar{v},\bar{a})$ determines which $l$-tuple is a vertex,
$\varphi_=(\bar{u},\bar{w},\bar{a})$ is an equivalence relation
defining equality between $l$-tuples and
$\varphi_\adj(\bar{x},\bar{y},\bar{a})$ defines the edge set of the
graph.

\begin{definition} {\label{def: Lindstrom, tuple sense}\ }
\begin{enumerate}
\item
Let $K$ be a graph property. The set of formulas
$\Logic^{\mathrm{tu}}(Q_k)$ is the closure of the atomic formulas by
first order operations together with additional quantifier $Q_K$.
The syntax of the new quantifier is given by $Q_K
\bar{v}\bar{u}\bar{w}\bar{x}\bar{y}. \psi(\bar{v},\bar{a}),
\varphi_=(\bar{u},\bar{w},\bar{a}),
\varphi_\adj(\bar{x},\bar{y},\bar{a})$ where:
\begin{enumerate}
\item all the vectors $\bar{v}, \bar{u}, \bar{w}, \bar{x}$ and $\bar{y}$ are of the same
length denoted henceforth by $l$;
\item $\varphi_\adj(\bar{x},\bar{y},\bar{a})
\in \Logic^{\mathrm{tu}}(Q_k)$ is an antireflexive and symmetric
formula in which $\bar{x},\bar{y}$ are $l$-tuples of free variables;
\item $\varphi_=(\bar{u},\bar{w},\bar{a}) \in \Logic^{\mathrm{tu}}(Q_k)$
is a reflexive, symmetric and transitive formula in which
$\bar{u},\bar{w}$ are $l$-tuples of free variables;
\item $\psi(\bar{v},\bar{a}) \in \Logic^{\mathrm{tu}}(Q_k)$ is a formula
in which $\bar{v}$ is an $l$-tuple of free variables; and
\item $\bar{a}$ are parameters.
\end{enumerate}

\item Given a graph $G=(V,E)$, the satisfaction of formulas of the
form $Q_K \bar{v}\bar{u}\bar{w}\bar{x}\bar{y}.
\psi(\bar{v},\bar{a}), \varphi_=(\bar{u},\bar{w},\bar{a}),
\varphi_\adj(\bar{x},\bar{y},\bar{a})$ is determined as follows. Let
$V' = \{\bar{v} \in V^l \mid \psi(\bar{v},\bar{a})\}$ and let $R$ be
the equivalence relation induced over $V'$ by $\varphi_=$. Then the
semantics of $Q_K$ in $\Logic^{\mathrm{tu}}(Q_K)$ is given by
\begin{align*}
&  G \models Q_K \bar{v}\bar{u}\bar{w}\bar{x}\bar{y}.
\psi(\bar{v},\bar{a}), \varphi_=(\bar{u},\bar{w},\bar{a}),
\varphi_\adj(\bar{x},\bar{y},\bar{a}) \Longleftrightarrow \\ %
&(V' / R , \{([\bar{x}],[\bar{y}]) \in (V' / R) \times (V' / R) \mid
\varphi_\adj(\bar{x},\bar{y},\bar{a})\}) \in K .
\end{align*}
\end{enumerate}
\end{definition}
As before, one must verify that the edges induced by $\psi_\adj$ are
well defined.

\subsection{Results} \label{subsec: results}
Our results are of mixed nature. We show that there are regular
languages able to express any first order sentence, connectivity and
$k$-colorability for any fixed $k$, and still obey the zero-one law
for $G(n,p)$ for any constant $0 < p < 1$. On the other hand, for
the same model of random graphs we show that in any semiregular
language able to express Hamiltonicity, there is a sentence with no
limiting probability.

Our results are concerned with connectivity, $k$-colorability and
Hamiltonicity, thus we define the following sets of graphs:
\begin{enumerate}
\item \Conn, the set of all connected graphs,
\item \Ham, the set of all Hamiltonian graphs, that is, the set of
graphs having a Hamilton cycle, and
\item \Chr{k}, the set of all graphs having chromatic number $k$.
\end{enumerate}
Now we can state our results.
\begin{theorem} \label{thm: conn+chr obey 0-1}
For every constant $0 < p < 1$ and for every sentence $\varphi$ in
$\Logic^{\mathrm{tu}}(Q_\Conn,Q_{\Chr{2}},Q_{\Chr{3}},\dotsc)$,
\[ \lim \Pr [G(n,p) \models \varphi] \in \{0,1\} . \]
\end{theorem}

\begin{remark}
Let $\Pln$ be the set of all planar graphs. Using the same technique
used for proving Theorem \ref{thm: conn+chr obey 0-1} one may add a
planarity Lindstr\"{o}m quantifier to the language appearing in the
statement of the theorem. The following result will be published
elsewhere:

For every constant $0 < p < 1$ and for every sentence $\varphi$ in
the graph language
$\Logic^{\mathrm{tu}}(Q_\Conn,Q_\Pln,Q_{\Chr{2}},Q_{\Chr{3}},\dotsc)$,
\[ \lim \Pr [G(n,p) \models \varphi] \in \{0,1\} . \]
\end{remark}

\begin{theorem} \label{thm: L(Q_ham) has arithmetic}
For any constant $0<p<1$, the pair $(\Logic(Q_\Ham), G(n,p))$ can
interpret arithmetic for some function $f = \Omega(\log \log \log
n)$.
\end{theorem}

\begin{remark} {\ }
\begin{enumerate}
\item The interpretation mentioned in Theorem \ref{thm: L(Q_ham) has
arithmetic} is theoretically explicitly given.
\item No effort was made on getting the best behavior for $f$, and
clearly $f = \Omega(\log \log \log n)$ is far from optimal. In fact,
with some effort one can show that the pair $(\Logic(Q_\Ham),
G(n,p))$ can interpret arithmetic for some \emph{linear} function.
\item Having arithmetization means that $\Logic(Q_\Ham)$ violates even the
modular limit law of Kolaitis and Kopparty.
\end{enumerate}
\end{remark}

The following corollary is an immediate consequence of Theorem
\ref{thm: L(Q_ham) has arithmetic} and the fact that
$\Logic(Q_\Ham)$ is the inclusion minimal semiregular language in
which Hamiltonicity is expressible.

\begin{corollary}
Let $\Logic$ be a language that can express Hamiltonicity and let
$0<p<1$ be constant, then the pair $(\Logic, G(n,p))$ has
arithmetization.
\end{corollary}
This answers the question of Blass and Harary
\cite{A:blass_&_harary1979}.

\subsection{Discussion}
This study suggests a notion of simplicity for properties of graphs.
Properties for which the \mbox{first order} closure obeys the
zero-one law are simpler than properties $P$ for which $\Logic(Q_P)$
has nonconverging sentences, or worse, are able to interpret
arithmetic. There is an intermediate level of properties for which
$\Logic(Q_P)$ does not obey the zero-one law but every sentence has
a limiting probability. Studying this question for other properties
of graphs seems natural and interesting. Obtaining general criteria
for graph properties to obey or fail the zero-one law seems within
grasp.

The same question arises also for other situations in which the
zero-one law holds. In particular we are interest in the behavior of
the same languages when the random graph is $G(n,p = n^{-\alpha})$
and $0<\alpha<1$ is irrational. As mentioned above, \FO\ obeys the
zero-one law in this case \cite{A:shelah_&_spencer88}. The partial
results that we have \cite{A:haber_&_shelahPPb} suggests that at
least $\Logic(Q_\Ham)$ behave differently at this situation, but
this is still a work in progress.

Looking at the proof of Theorem \ref{thm: conn+chr obey 0-1} one
sees that while the language has strictly stronger expressive power,
it defines the same family of definable sets as \FO. We were also
interested in the question of finding a language with more definable
sets than \FO, that still obeys the zero-one laws. Lately a rather
natural language having these properties was found
(\cite{A:haber_&_shelahPPa}, paper in preparation).

\subsection{Notation}
Along the paper we use standard graph theory notions and notation.

Let $G=(V,E)$ be a graph. We denote the neighborhood of a vertex $v$
by $N(v) = \{u\in V | u \adj v\}$. We denote the degree of $v$ by
$d(v) = |N(v)|$. Given a vertex set $S \subset V$ we denote the
neighborhood of $v$ in $S$ by $N(v,S) = N(v) \cap S$ and the degree
of $v$ in $S$ by $d(v,S) = |N(v,S)|$. Given two vertices $u,V$ their
\emph{codegree} is the number of common vertices, and is denoted
here by $\codeg(u,v)$. We refer the reader to \cite{B:west2000},
\cite{B:bollobas02} or \cite{B:diestel2010} for general graph theory
monographs.

Let $\Logic$ be a formal language for the class of graphs and let
$G=(V,E)$ be a graph. A vertex set $A\subset V$ is said to be
\emph{definable in $\Logic$} (also \emph{$\Logic$-definable}), if
there is a formula $\varphi(x) \in \Logic$ with $x$ being a free
variable such that $A = \{ a \mid G \models \varphi(a) \}$.

We use the standard ``Big O'' asymptotic notation of Bachmann and
Landau. Let $f(n),g(n)$ be two positive functions whose domain is
$\omega = \{0,1,2,\dotsc\}$. We say that $f=O(g)$ if there is a
constant $C$ such that $f(n) \leq Cg(n)$ for every integer $n$. We
say that $f=\Omega(g)$ if $g=O(f)$, and that $f=\Theta(g)$ if both
$f=O(g)$ and $f=\Omega(g)$. If $f/g \rightarrow 0$ as
$n\rightarrow\infty$ we write $f=o(g)$, if $f=o(g)$ then we also
write $g=\omega(f)$. In particular, we may use $f=o(1)$ and
$f=\omega(1)$ to denote functions tending to zero and to infinity
respectively. We also write $f=(1\pm\epsilon)g$ if there exists a
constant $N$ such that $(1-\epsilon)g(n) \leq f(n) \leq
(1+\epsilon)g(n)$ for every $n>N$. We add a diacritical tilde above
the relation symbol to denote asymptotic relation which holds up to
polylogarithmic factors. That is, $f = \tilde{O}(G)$ means that $f =
O(g \cdot h)$ where $h$ is some polylogarithmic function. In the
case of the strict relations the tilde means ``by more than a
polylogarithmic function''. For example, $f = \tilde{\omega}(g)$
means that $f/gh \rightarrow \infty$ for \emph{any} polylogarithmic
$h$.

We use $\ln$ for base $e=2.718281828\dotsc$ logarithms and $\log$
for base two logarithms.

\section{Connectivity and chromatic number}
In this section we prove Theorem \ref{thm: conn+chr obey 0-1}.
\begin{proof}[Proof of Theorem \ref{thm: conn+chr obey 0-1}]
Let $p$ be a constant, $0 < p < 1$. Assume first that $\varphi$ is a
sentence in the language $\Logic =
\Logic(Q_\Conn,Q_{\Chr{2}},Q_{\Chr{3}},\dotsc)$, which is the
first-order closure of the quantifiers
$Q_\Conn,Q_{\Chr{2}},Q_{\Chr{3}},\dotsc$ (using the weakest sense of
a Lindstr\"om quantifier as in Definition \ref{def: Lindstrom Q,
Flum's variant}). We wish to show that the limiting probability of
$\varphi$ in $G(n,p)$ is either zero or one. We use induction on the
structure of $\varphi$.

{\bf Claim \ref{thm: conn+chr obey 0-1}.1:} Let $\varphi(\bar{a})$
be a formula in $\Logic$, where $\bar{a}$ are free variables (seen
as parameters of $\varphi$). Then the limiting probability of
$\varphi(\bar{a})$ in $G(n,p)$ with $0<p<1$ is either zero or one,
depending only on the relations among $\bar{a}$.

{\bf Proof of Claim \ref{thm: conn+chr obey 0-1}.1:} By induction on
the structure of the formula. For an atomic formula it is clear
(that is, the truth value of $a_1=a_2$ and $a_1 \sim a_2$ depends
only on the relations among $a_1$ and $a_2$...).

If $\varphi$ is of the form $\neg \psi$ or $\psi_1 \wedge \psi_2$
then the claim is immediate.

Assume that $\varphi(\bar{a})$ is of the form $\varphi(\bar{a}) = Q
\bar{y}. \psi(\bar{a},\bar{y})$ where $Q$ is one of $\exists,
Q_\Conn, Q_{\Chr{k}}$ and $\bar{y}$ is of length one or two.
Consider the equivalence relation $E_{\bar{a}}$ over vertices
defined by
\[ x E_{\bar{a}} y \Longleftrightarrow \forall a \in \bar{a}.\, \left[ (x = a
\leftrightarrow y = a) \wedge (x \sim a \leftrightarrow y \sim a)
\right] ,
\]
and let $B_i, 1 \leq i \leq k_* = |\bar{a}| + 2^{|\bar{a}|}$ be the
equivalence classes in $[n] / E_{\bar{a}}$. By the induction
hypothesis, the limiting probability of $\psi(\bar{a},\bar{y})$
depends only on the relations among $\bar{a}$ and $\bar{y}$. In
other words, if $Q = \exists$ then the aforementioned limiting
probability is determined by the relations among $\bar{a}$ and the
$i$ for which $b \in B_i$. If $Q = Q_\Conn$ or $Q = Q_{\Chr{k}}$
then the limiting probability depends on relations among $\bar{a}$,
the equivalence classes containing the two variables of $\bar{y}$
and the relations between these variables.

Now, for the first time, probability enters the proof. Each of the
singletons $\{a_i\}$ is an equivalence class. Let $B_i$ be an
equivalence class defined by some adjacency pattern between a vertex
and $\bar{a}$, and let $t=t(i)$ be the number of adjacencies in the
pattern. The probability that a given vertex belongs to $B_i$ is
then $p_* = p_*(i) \defeq p^t(1-p)^{|\bar{a}| - t }$. Clearly and
importantly, the event ``$v$ belongs to $B_i$'' is independent of
all other events of the form ``$u$ belongs to $B_j$''. Hence the
size of the equivalence class follows a binomial distribution with
parameters $n-|\bar{a}|$ and $p_*$, and in particular a.a.s.\ all
equivalence classes are of size which is tightly concentrated around
its linear sized ($p_* (n-|\bar{a}|)$) mean.

Assume first that $Q$ is the existential quantifier. If for any of
the (finitely many) $B_i$'s the limiting probability of
$\psi(\bar{a},y)$ is one for $y \in B_i$, then $\varphi$ also has
limiting probability one. Otherwise $\varphi$'s limiting probability
is zero.

Next assume that $\varphi(\bar{a})$ is of the form $\varphi(\bar{a})
= Q_\textsc{Conn} y_1,y_2. \psi(\bar{a},y_1,y_2)$. We want to show
that the limiting probability for the connectivity of the graph $F =
([n], \{\{y_1,y_2\} \mid \psi(\bar{a},y_1,y_2)\} )$ depends solely
on the relations between the members of $\bar{a}$ and themselves. If
$\bar{a}$ is the empty sequence, then $F$ is either the empty graph
over $[n]$ (which is disconnected), the complete graph over $[n]$
(connected), the original graph or its compliment. In the last two
cases, since the original graph is sampled from $G(n,p)$, it is
a.a.s.\ connected (this is a well known fact. See, e.g.,
\cite{B:janson_&_luczak_&_rucinski00}). All in all, the limiting
probability is either zero or one as required.

Assume now that $\bar{a}$ is not empty. Notice that by the induction
hypothesis, once $\bar{a}$ and the relations among its members are
given, the limiting probability of $\psi(\bar{a},y_1,y_2)$ depends
only upon the classes of $[n] / E_{\bar{a}}$ containing $y_1$ and
$y_2$, whether $y_1 \sim_{G(n,p)} y_2$ and whether $y_1 = y_2$. In
particular, if both $y_1$ and $y_2$ are not singleton, then we are
in one of four situations. It may be that for every $y_1 \in B_1,
y_2 \in B_2$ the formula $\psi(\bar{a},y_1,y_2)$ is true (for given
$\bar{a}$), or that for every such pair the formula
$\psi(\bar{a},y_1,y_2)$ is false, or that the truth value of
$\psi(\bar{a},y_1,y_2)$ depends \emph{solely} on the truth value of
$y_1 \sim_{G(n,p)} y_2$ (the last case counts as two, either
$\psi(\bar{a},y_1,y_2)$ agrees with $y_1 \sim_{G(n,p)} y_2$ or they
disagree).

Consider the graph $H$ whose vertex set is the quotient set $[n] /
E_{\bar{a}}$ and its edge set is determined by
\[ B_1 \sim_H B_2 \Longleftrightarrow \exists y_1 \in B_1, y_2 \in B_2
.\, \psi(\bar{a},y_1,y_2) . \] %
Notice that we allow loops in $H$. We argue that if $H$ is
disconnected then so is $F$, and if $H$ is connected then $F$ is
connected with probability tending to one with $n$. Let $u$ and $w$
be two vertices of $F$ and assume that there is a path
$v_0=u,v_1,\dotsc,v_l = w$ in $F$ connecting $u$ and $w$. Then
$[v_0], [v_1], \dotsc, [v_l]$ is a walk connecting $u$ and $w$ in
$H$ (formally this may not be a walk since it may be that $[v_i] =
[v_j]$ without a loop. Still it contains a path connecting $[v]$ and
$[w]$). Hence, if $H$ is disconnected then $F$ is disconnected as
well.

Assume now that $H$ is connected and let $u,w$ be two vertices in
$F$. For the moment, assume also that $[u] \neq [w]$. Denote the
vertices of the shortest path connecting $[u]$ and $[w]$ in $H$ by
$[v_0=u],[v_1],\dotsc,[v_l = w]$. We find a path connecting $u$ and
$w$ in $F$ by induction on the length of the path as follows. Start
by denoting $u$ as $v'_0$ and assume that we have a path in $F$
connecting $u$ to $v'_i$ where $v'_j \in [v_j]$ for every $0 \leq j
\leq i$. Since $[v_i] \sim_H [v_{i+1}]$, we know that there are two
vertices $y_1 \in [v_i]$ and $y_2 \in [v_{i+1}]$ such that
$\psi(\bar{a}, y_1, y_2)$ holds. If either $[v_i]$ or $[v_{i+1}]$ is
a singleton (or both), then, by definition of $H$'s vertices, all
the vertices of $[v_i]$ relate in the same manner to all the
vertices of $[v_{i+1}]$. Hence we can take $v'_{i+1} = v_{i+1}$. If
both $[v_i]$ and $[v_{i+1}]$ are not singletons, then, as argues
above, both are of size linear in $n$. If the value of
$\psi(\bar{a},y_1,y_2)$ does not depend on $y_1 \sim_{G(n,p)} y_2$,
we can simply pick $v'_{i+1} = v_{i+1}$. Otherwise we use the fact
that the probability of $v'_i$ not having a neighbor or a
non-neighbor in a vertex set of size $cn$ is exactly $p^{cn} +
(1-p)^{cn} \leq c'^n$ where $0 < c' < 1$ is some constant. Hence,
with probability tending to one exponentially fast we can find a
vertex $v'_{i+1} \in [v_{i+1}]$ satisfying
$\psi(\bar{a},v'_{i},v'_{i+1})$. When picking $v'_{l-1}$ we need to
take extra care and make sure that it is a neighbor of $w$ in $F$ as
well. This is easily done using a similar consideration (and similar
computations).

Assume now that $[u] = [w]$. In this case we consider a closed walk
instead of the shortest path. That is, we denote $[v_0=u],[v_1],[v_2
= w]$, where $[v_1] \neq [u]$. Such $[v_1]$ exists since we are in
the case of nonempty $\bar{a}$, meaning that the number of vertices
in the connected graph $H$ is greater than one. The rest of the
proof is identical to the above.

We have shown that with probability tending exponentially fast to
one, the connectivity of $F$ depends only on the connectivity of
$H$. By the definition of $H$, it is connected depending only on
$\bar{a}$.

Finally let $Q = \Chr{k}$ for some constant $k \in \mathbb{N}$. We
look again on the graphs $F$ and $H$ defined above. By arguments
similar to the above, there are $k_*$ vertices $v_1,\dotsc,v_{k_*}
\in F$ such that $v_i \in B_i$ and $v_i \sim_F v_j \Leftrightarrow
B_i \sim_H B_j$. Hence $F$ has a copy of $H$ as a subgraph and
therefore $\chi(F) \geq \chi(H)$. We argue that with probability
tending to one there are two possibilities: either $\chi(F)
=\chi(H)$ or $\chi(F) = \Omega(n / \ln n)$. If for every $i$ the
vertices of $B_i$ form an independent set in $F$, then we can color
$F$ by assigning the color of $B_i$ to all its vertices. Hence in
this case $\chi(F) =\chi(H)$. Assume now that there is a set $B_i$
that is not an independent set and let $n' = |B_i|$ and recall that
a.a.s.\ $n' = cn(1+o(1)))$. The graph spanned by $F$ on $B_i$ is
either a clique, a spanned subgraph of $G(n,p)$ or the complement of
a spanned subgraph of $G(n,p)$. In the first case the chromatic
number of the spanned graph is $n'$, so $\chi(F) = \omega(n/\log
n)$. In the second and third cases the chromatic number of $F[B_i]$
is in fact the chromatic number of a random graph with $n'$ vertices
and edge probability $p'$ being equal to $p$ or to $1-p$. In an
exciting paper \cite{A:bollobas1988b} Bollob\'{a}s showed that with
probability tending to one this number is $(1/2 + o(1)) n'/\log_b
n'$ where $b = 1/(1-p')$. Thus in this case the chromatic number of
$F$ is of order $n / \ln n$.

Summarizing the above argument we get that the chromatic number of
$([n], \{\{y_1,y_2\} \mid \psi(\bar{a},y_1,y_2)\} )$ depends only on
the relations among $\bar{a}$ and it is either a specific constant
or growing to infinity with $n$. Therefore the limiting probability
of $Q_{\Chr{k}} xy .\, \psi(\bar{a},x,y)$ is either zero or one,
depending only on the relations among $\bar{a}$.
\newline{\bf End of proof of Claim {\ref{thm: conn+chr obey 0-1}.1}}.

The theorem for the case of formulas in $\Logic$ follows immediately
as $\varphi$ is a sentence in $\Logic$ with no free variables.

We argue now that the proof works, mutatis mutandis, also for
$\varphi \in \Logic^{\mathrm{tu}} =
\Logic^{\mathrm{tu}}(Q_\Conn,Q_{\Chr{2}},Q_{\Chr{3}},\dotsc)$, where
$\Logic^{\mathrm{tu}}(Q_\Conn,Q_{\Chr{2}},Q_{\Chr{3}},\dotsc)$ is
the first order closure of $Q_\Conn,Q_{\Chr{2}},Q_{\Chr{3}},\dotsc$
with the syntax and semantics as in Definition \ref{def: Lindstrom,
tuple sense}. The proof still works by quantifier elimination, and
the key observation is that we can still partition the vertices into
equivalence classes that are either singletons or sets of polynomial
size. We restrict ourselves to a description of the changes in the
induction step in Claim \ref{thm: conn+chr obey 0-1}.1. Assume that
$\varphi$ is of the form $\varphi = Q_K
\bar{v}\bar{u}\bar{w}\bar{x}\bar{y}. \psi(\bar{v},\bar{a}),
\varphi_=(\bar{u},\bar{w},\bar{a}),
\varphi_\adj(\bar{x},\bar{y},\bar{a})$, where $k$ is one of
$Q_\Conn,Q_{\Chr{2}},Q_{\Chr{3}},\dotsc$, and $\psi, \varphi_=$ and
$\varphi_\adj$ are as in Definition \ref{def: Lindstrom, tuple
sense}. Let $l = |\bar{v}|$, and notice the number $l$-tuples of
vertices is $n^l$. The relations between the $l$-tuples and the
elements of $\bar{a}$ will divide the set of $l$-tuples into
$(|\bar{a}| + 2^{|\bar{a}|})^l$ --- finitely many --- equivalence
classes, denoted as above by $B_i$. The size of an equivalence class
defined by $s$ equalities to elements of $\bar{a}$ (and thus
$|\bar{a}|-s$ inequalities --- the adjacencies does not matter here)
is $\Theta(n^{l-s})$. To see this, consider an equivalence class
$B_i$ defined by $s$ equalities to elements of $\bar{a}$. There are
$(n-|\bar{a}|)^{l-s}$ tuples\footnote{If some of the elements of
$\bar{a}$ are identical, we may have a few more candidate tuples.}
that satisfy the $s$ equalities. Call the variables that are not
required to be equal to any of the elements of $\bar{a}$ an
inequality variable. Then the number of $l$-tuples that additionally
satisfy all the adjacency relations for the first inequality
variable is a Binomial random variable with parameters
$(n-|\bar{a}|)^{l-s}$ and $p^*$, where $p^*$ is a constant of the
form $p^r(1-p)^t$ ($r$ and $t$ are two constants satisfying $0 < r,t
< |\bar{a}|$). Hence a.a.s.\ the number of such $l$-tuples is of
order $n^{l-s}$. Applying the same argument to the rest of the
(finitely many) inequality variables, gives that $B_i$ is of size
$\Theta(n^{l-s})$.

Now, $\psi$ defines the set of vertices, and it will be a union of
equivalence classes, hence it is either of constant size or of
polynomial size in $n$. The affect of $\varphi_=$ is similar.
Indeed, $\varphi_=(\bar{x}, \bar{y})$ depends on the classes $B_i,
B_j$ that contain $\bar{x}$ and $\bar{y}$, but also on the relations
between $\bar{x}$ and $\bar{y}$. Still, there are only finitely many
options for such relations, and an argument identical to the
argument above gives that the number of $l$-tuples in each vertex of
the defined graph remains either constant or polynomial in $n$. The
rest of the proof remains very similar.
\end{proof}

\section{Hamiltonicity}
In this section we prove Theorem \ref{thm: L(Q_ham) has arithmetic}
by showing that using properties of $G(n,p)$ we can encode any
\emph{second order} sentence into a sentence of $\Logic(Q_\Ham)$
such that the sentences are equivalent on a set of vertices with
size tending to infinity at a controlled pace as $n$ grows.

Let us start with a brief overview of the proof. Using the $Q_\Ham$
quantifier we shall be able to express $|A| \leq |B|$ where $A$ and
$B$ are (definable) vertex sets. Next we would like to find sets
that are: definable, small enough so that all their subsets are
definable, and finally also large enough --- tend to infinity as $n$
grows in an appropriate rate. Given the ability to express equality
of definable sets it seems plausible to look on sets of the form
``all vertices having degree $d$''. Still, it is not obvious how to
continue, as the set of all vertices of degree $np$ is too large,
while the set of vertices of, say, minimal degree is too small. We
shall define such a set having size $\Theta(\log \log n)$. This
enables us to interpret \emph{monadic} second order on this set,
which almost suffice considering the fact that the induced graph on
this set should be quite random and then an application of
\cite{A:Kaufmann_&_Shelah1985} should give arithmetization. We shall
prefer to give a direct proof that will require less randomness from
the induced graph, and it is therefore technically less involved.

\newcommand{\urep}{powerset representative}
We begin by defining two graph properties that are needed later for
the encoding scheme.
\begin{definition}
Let $\Logic$ be a language and let $\varphi(x,\bar{y})\in\Logic$ be
a formula in which $x$ is a free variable. We say that a given graph
$G=(V,E)$ \emph{allows $\varphi$-powerset representation}, if there
exists a vertex sequence $\bar{b} \in V^{|\bar{y}|}$ such that for
every subset $A \subset S
\defeq \{ x \in V \mid G \models \varphi(x,\bar{b})\}$ there is a vertex
$v\in V$ such that $N(v,S) = A$.
\end{definition}
If $G$ allows $\varphi$-powerset representation then we can
represent monadic second order variables on $S$ by vertices from
$V$. Generally, this is useful if $S$ is large enough. In order to
encode second order sentences we need a bit more:
\newcommand{\SOmim}{double powerset representative}
\comment { %
@ This is the definition that Saharon sent me in the scanned pages
from 18-July-2011 (or what I was able to get from it). @
\begin{definition} \label{def: SOmim}
Let $\Logic$ be a language of graphs. We say that a graph $G =
(V,E)$ \emph{allows $(\varphi_0(x,\bar{y}_0),
\varphi_1(x,\bar{y}_1), f_{l,0}, f_{l,1})$-double powerset
representation} when
\begin{enumerate}
\item $f_0, f_1 \colon \mathbb{N} \to \mathbb{N}$, and $\forall n.\,
f_0(n) \leq f_1(n) \leq n$.
\item \label{item: b_0}
There is a finite sequence $\bar{b}_0 \in V^{|\bar{y}_0|}$
such that $S_{\bar{b}_0} = \{a \mid G \models
\varphi_0(a,\bar{b}_0)\}$ is of size at least $f_0(|V|)$.
\item For every $\bar{b}'_0 \in V^{|\bar{y}_0|}$ one has $|\{a \mid G \models
\varphi_0(a,\bar{b}'_0)\}| \leq f_1(|V|)$.
\item Let $S = S_{\bar{b}_0}$ for $\bar{b}_0$ as in Item \ref{item: b_0}
above. Then
\begin{enumerate}
\item For every $A \subseteq S$ there is $v \in V$ such that $N(v,S)
= A$.
\item There is a finite sequence $\bar{b}_1 \in V^{|\bar{y}_1|}$
such that $B_{\bar{b}_1} = \{a \mid G \models
\varphi_0(a,\bar{b}_0)\}$ has at least $f_1(|S|)$ members.
\item If $\bar{b}_1 \in V^{|\bar{y}_1|}$ is such that
$B_{\bar{b}_1}$ is maximal, then for every $A \subseteq
B_{\bar{b}_1}$ there is a vertex $v \in S_{B_{\bar{b}_0}}$ such that
$N(v,S_{B_{\bar{b}_1}}) = A$.
\end{enumerate}
\end{enumerate}
\end{definition}
} %
\begin{definition} \label{def: SOmim}
Let $\Logic$ be a language of graphs. We say that a graph $G =
(V,E)$ \emph{allows $(\varphi_0(x,\bar{y}_0),
\varphi_1(x,\bar{y}_1), f)$-double powerset representation} when
\begin{enumerate}
\item $f \colon \mathbb{N} \to \mathbb{N}$ is a monotone increasing
function.
\item \label{item: b_0}
There exists a vertex sequence $\bar{b}_0 \in V^{|\bar{y}|}$ such
that for every subset $A \subset S_{\bar{b}_0}
\defeq \{ x \in V \mid G \models \varphi(x,\bar{b}_0)\}$ there is a vertex
$v\in V$ such that $N(v,S) = A$. That is, $G$ allows
$\varphi_0$-powerset representation.
\item Let $S = S_{\bar{b}_0}$ for $\bar{b}_0$ as in Item \ref{item: b_0}
above. Then
\begin{enumerate}
\item There is a finite vertex sequence $\bar{b}_1 \in V^{|\bar{y}_1|}$
such that the set $B_{\bar{b}_1} \defeq \{x \in S \mid G \models
\varphi_1(x,\bar{b}_1)\}$ has at least $f(|V|)$ members.
\item For every $A \subseteq
B_{\bar{b}_1}$ there is a vertex $v \in S$ such that $N(v,S) = A$.
\end{enumerate}
\end{enumerate}
\end{definition}
Notice that in the definition above every subset of $B\defeq
B_{\bar{b}_1}$ is represented as a neighborhood in $B$ of a vertex
from $S$. The set $B$ will be the set over which we shall be able to
encode second order sentences. Hence we are interested in a lower
bound for its size, which is the role of $f$.
\comment{ %
@ This is my original definition for \SOmim. It is not too good. @
\begin{definition} \label{def: (old) SOmim}
Let $\Logic$ be a language of graphs. We say that a graph $G =
(V,E)$ is a \emph{\SOmim} if there exists a set $S$
defineable\comment{\footnote{That is, there exists a formula
$\varphi(v) \in \Logic$ where $v$ is a free variable, such that $S =
\{v\in V | \varphi(v) \}$.}} in $\Logic$ such that
\begin{enumerate}
\item For every subset $A\subset S$ there exists a vertex $v\in V$
such that $N(v,S) = A$. (That is, $S$ is \urep\ with respect to
$G$).
\item There exists a set $B \subset S$ such that $B$ is definable in
$\Logic$ and for every subset $A \subset B$ there exists a vertex
$v\in S$ such that $N(v,B) = A$. (That is, $B$ is \urep\ with
respect to $G[S]$).
\end{enumerate}
\end{definition}
} %

The following lemma demonstrates the usefulness of the definition
above --- basically saying that when a graph $G$ on $n$ vertices
allows $(\varphi_0, \varphi_1, f)$-double powerset representation,
then we can encode arithmetic over the set $\{1,\dotsc,f(n)\}$. This
lemma uses ideas similar to the main ideas of
\cite{A:Kaufmann_&_Shelah1985}.
\begin{lemma} \label{lem: 1. G SOmim -> encode 2nd order}
Let $\Logic$ be a semiregular language and let $G=(V,E)$ be a graph.
Assume that $G$ allows $(\varphi_0, \varphi_1, f)$-double powerset
representation. Then there is a set $B$ of size at least $f(|V|)$
such that for every second order sentence $\varphi$ there is a
sentence $\varphi' \in \Logic$ satisfying
\[ G[B] \models \varphi \Longleftrightarrow G \models \varphi' . \]
\end{lemma}
\begin{proof}
Let $G = (V,E)$ be a graph that allows $(\varphi_0(x,\bar{y}_0),
\varphi_1(x,\bar{y}_1), f)$-double powerset representation and let
$\bar{b}_0$ and $\bar{b}_1$ be the vertex sequences for which
Definition \ref{def: SOmim} is realized. Denote by $S$ and $B$ the
sets $S_{\bar{b}_0}$ and $B_{\bar{b}_1}$ (respectively) appearing in
the definition. We start by defining a linear order over the
vertices of $B$. Fix some linear order on $B$ and mark the vertices
of $B$ according to that order by $1, 2, \dotsc, b$. Since $G$
allows $(\varphi_0, \varphi_1, f)$-double powerset representation,
for every $1 \leq i \leq b$ there is a vertex $v_i$ in $S$ such that
$N(v_i, B) = \{1, \dotsc, i\}$. Also, again by $G$ allowing a
$(\varphi_0, \varphi_1, f)$-double powerset representation, there is
a vertex $v_<$ such that $N(v_<, S) = \{v_1, \dotsc, v_b\}$. We
shall use vertices like $v_<$ to represent a linear order over $B$.

We want a vertex $v_< \in V$ satisfying the formula
$\predfont{Order}(v_<)$ defined and explained immediately:
\begin{align*}
\predfont{Order}(v_<) \defeq & [ \forall s,t \in S.\, (s \sim v_<
\wedge t \sim v_<) \rightarrow \\ %
& ((N(s,B) \subsetneq N(t,B)) \vee (N(t,B) \subsetneq N(s,B)) ) ]
\wedge \\ %
& \forall b \in B.\, \exists s,t \in N.\, [(v_<). N(t,B) \triangle
N(s,B) = \{b\} ] .
\end{align*}
That is, for $\predfont{Order}(v_<)$ to hold for some vertex $v_<
\in V$ the following must happen:
\begin{enumerate}
\item Ordering the neighbors of $v_<$ in $S$ according to inclusion
of their neighborhoods in $B$ gives a linear order.
\item The number of elements in that order ($d(v_<,S)$) is at least
$|B|$ as any element of $B$ separates between the neighborhoods of a
pair of $v_<$ neighbors. (Also $d(v_<,S) \leq |B|+1$).
\end{enumerate}
Clearly $v_<$ induces a linear ordering over the elements of $B$.
Let us give a concrete definition. Assume that $v_< \in V$ satisfies
$\predfont{Order}(v_<)$. We define
\[ b_1 <_{v_<} b_2 \defeq \forall s \in S.\, [((s \sim v_<) \wedge (s
\sim b_2)) \rightarrow (s \sim b_1)] .\] %
It is not difficult to observe that $<_{v_<}$ is a linear order over
$B$. Again, since $G$ allows $(\varphi_0, \varphi_1, f)$-double
powerset representation we can find such $v_<$ for every linear
order $<$ over $B$.

Having order we can describe the encoding of second order sentences.
Let $\varphi$ be a second order sentence. We want to write a
sentence $\varphi' \in \Logic$ such that $G[B] \models \varphi
\Leftrightarrow G \models \varphi'$.\comment{Assume for the moment
that all the second order variables in $\varphi$ are either unary or
binary.} Monadic variables can be replaced using vertices from $S$
by replacing, e.g. $\exists A^1.\,\psi(A)$ by $\exists v_A\in
S.\,\psi'(v_A)$ where $\psi'$ is the sentence we get by replacing
every occurrence of $x \in A$ in $\psi$ by $x \sim v_A$.

For binary variables we need a bit more. Since $G$ allows
$(\varphi_0, \varphi_1, f)$-double powerset representation, for
every pair of vertices $x,y \in B$ we have a vertex $s_{\{x,y\}} \in
S$ such that $N(s_{\{x,y\}},B) = \{x,y\}$. Every binary relation $R$
can be represented by two sets of pairs
--- the set of all the pairs $x,y$ such that $x \leq_{v_<} y$ and
$xRy$ and the set of all the pairs $x,y$ such that $x >_{v_<} y$ and
$xRy$. We represent binary relations using this idea. That is, we
replace every occurrence of the form $\exists R^2.\,\psi(R)$
appearing in $\varphi$ by $\exists v_{R,\{1,2\}}, v_{R,\{2,1\}}.\,
\psi'(v_{R,\{1,2\}}, v_{R,\{2,1\}})$. In order to get
$\psi'(v_{R,\{1,2\}}, v_{R,\{2,1\}})$ from $\psi(R)$ we replace
occurrences of the form $xRy$ in $\psi$ by
\begin{align*}
& ((x \leq_{v_<} y) \wedge (\overline{\exists} s_{\{x,y\}}.\,
[s_{\{x,y\}} \adj
v_{R,\{1,2\}}])) \vee \\
& ((x >_{v_<} y) \wedge (\overline{\exists} s_{\{x,y\}}.\, [
s_{\{x,y\}} \sim v_{R,\{2,1\}}])) , %
\end{align*}
where we use $\overline{\exists} s_{\{x,y\}}.\,
\varphi(s_{\{x,y\}})$ as an abbreviation standing for $\exists
s_{\{x,y\}}.\, [s_{\{x,y\}} \in S \wedge N(s_{\{x,y\}}, B) =
\{x,y\}]$.

Generally, we represent $k$-ary relation variables in a similar
manner. That is, letting $R^k$ be a $k$-ary relation, we think of
$R$ as a family of $k!$ collections of sets of $k$ vertices from
$B$. For each such $k$-set $\{x_1,\dotsc,x_k\}$ there is a vertex
$s_{\{x_1,\dotsc,x_k\}}$ in $S$ such that $N(s_{\{x_1,\dotsc,x_k\}},
B) = \{x_1,\dotsc,x_k\}$ (by $G$ allowing a
$(\varphi_0(x,\bar{y}_0), \varphi_1(x,\bar{y}_1), f)$-double
powerset representation). Now consider a $k$-tuple $\bar{x} =
(x_1,\dotsc,x_k) \in R$, and let $\pi = \pi(\bar{x})$ be the
permutation such that $x_{\pi(1)} \leq x_{\pi(2)} \leq \dotsb \leq
x_{\pi(k)}$ where in case of equality $\pi$ maintains the original
order (that is, if $x_i = x_j$ and $i<j$ then $\pi(i) < \pi(j)$).
Let $X_\pi$ be the collection of all $k$-tuples $\bar{x}\in R$ such
that $\pi(\bar{x}) = \pi$. We represent $X_\pi$ by a vertex $v_\pi
\in V$ with the property that $N(v_\pi,S) = \{
v_{\{x_1,\dotsc,x_k\}} \mid (x_{\pi^{-1}(1)}, x_{\pi^{-1}(2)},
\dotsc, x_{\pi^{-1}(k)}) \in R \}$ (such a vertex exists by our
assumption on $G$). Finally, $R$ is represented by the family
$\{v_\pi\}$ where $\pi$ ranges over all the permutations of
$\{1,2,\dotsc,k\}$. This completes the proof.
\end{proof}

Next we aim to show that a.a.s.\ $G(n,p)$ allows a
$(\varphi_0(x,\bar{y}_0), \varphi_1(x,\bar{y}_1), f)$-double
powerset representation with respect to two formulas in $Q_\Ham$
with $f$ growing to infinity quickly enough. The first formula
defines a set of vertices all having the same degree. That is,
$\varphi_0(x,v)$ will simply express the fact that $d(x) = d(v)$
(the details of writing $\varphi_0$ in $\Logic(Q_\Ham)$ are given in
the beginning of the proof of Theorem \ref{thm: L(Q_ham) has
arithmetic} in Page \pageref{proof: L(Q_ham) has arithmetic}). Thus,
for an integer $m$ we define $D(m) = \{v\in V \mid d(v) = m\}$.
Somewhat abusing notation we also use $D(v)$ (for a given vertex
$v$) for the set of all vertices with the same degree as $v$, that
is, $D(v) = D(d(v))$. The second formula we shall use is even
simpler. $\varphi_1$ will define a neighborhood of a vertex. That
is, $\varphi_1(x,u) \defeq x\adj u$. Lemma \ref{lem: 2. G(n,p) is
SOmim w.r.t. D(m) and N(u,D(m))} shows that a.a.s.\ there are
two vertices $v,u$ such that $G(n,p)$ allows a $(\varphi_0(x,v),
\varphi_1(x,u), f)$-double powerset representation with $f =
\Theta(\log \log \log n)$.

The main idea of the proof of Lemma \ref{lem: 2. G(n,p) is SOmim
w.r.t. D(m) and N(u,D(m))} is to show that $D(m)$ (with $m$ properly
defined) is pseudorandom, in a sense to be given soon. Then we shall
use this pseudorandomness to show that there is a subset fitting to
the requirement of Definition \ref{def: SOmim} (and every such
subset is definable by $\varphi_1(x,u)$ for some vertex $u$).
\comment{Now, $D(m)$ is not necessarily definable, but we may take
the \emph{maximal} set (by cardinality) for which the arithmetic
interpretation holds. This will give a definable set, with a
controlled cardinality, on which we have arithmetic.}

We begin by picking a suitable degree. Let
\begin{equation} \label{eq: definition of m and h}
m = np + h = np + \sqrt{n\left( \frac{1}{2}\ln n -\ln\ln\ln n +
\alpha \right) 2p(1-p)} ,
\end{equation}
where $\alpha$ is $O(1)$.

Recall that $\varphi_0(x,v)$ is a formula in $\Logic(Q_\Ham)$
satisfying $G \models \varphi_0(x,v) \Longleftrightarrow x \in D(v)$
and $\varphi_1(x,u) = x \adj u$.
\begin{lemma} \label{lem: 2. G(n,p) is SOmim w.r.t. D(m) and N(u,D(m))}
Asymptotically almost surely there exist two vertices $u,v$ in
$G(n,p)$ such that $G(n,p)$ allows a $(\varphi_0(x,v),
\varphi_1(x,u), f)$-double powerset representation where $f =
\Theta(\log\log\log n)$.
\end{lemma}
Before proving Lemma \ref{lem: 2. G(n,p) is SOmim w.r.t. D(m) and
N(u,D(m))} we quote a fact regarding the distribution of $|D(m)|$ in
$G(n,p)$. This distribution was studied a few times (e.g.,
\cite[Chapter 3]{B:bollobas01},
\cite{B:barbour_&_holst_&_janson1992} and
\cite{A:mckay_&_wormald1997}). For our application we cite (part of)
Theorem 5.F of \cite{B:barbour_&_holst_&_janson1992}, which bounds
the \emph{total variation distance} between the distribution of the
number of vertices of degree $m$ and a Poisson distribution (The
total variation distance between two probability measures $\lambda$
and $\mu$ (both over the set of natural numbers) is defined as $\sup
\left\{ |\lambda(A) - \mu(A)| \middle\vert A \subset \omega
\right\}$). We added the pace of decay of the total variation,
originally appearing only in the proof, to the statement of the
theorem.
\begin{theorem}[{\cite[Theorem 5.F]{B:barbour_&_holst_&_janson1992}}] \label{thm: Dist of |{v|d(v) = m}|}
Let $W$ be a random variable counting the number of vertices of
degree $m=m(n)$ in $G(n,p)$. Assume that $np$ is bounded away from
zero and  that $(np)^{-1/2}|m-np| \rightarrow \infty$. Then the
total variation distance between the distribution of $W$ and a
Poisson distribution with parameter $\lambda = n \cdot
\Pr[\Bin(n-1,p) = m]$ tends to zero with $n$. Moreover, the total
variation distance is bounded by
\[ \frac{Cm (m-np)}{np} \Pr[\Bin(n-1,p) = m] , \]
where $C$ is some constant independent of $n$.
\end{theorem}
Next we quote a Chernoff-type concentration theorem regarding the
tails of a Poisson distribution. The last inequality in the
statement is due to the fact that $\ln(1+x) \geq x - \frac{1}{2}x^2$
for $x>0$.
\begin{theorem}[{\cite[Theorem A.1.15]{B:alon_&_spencer08}}]
\label{thm: Pois is concentrated around its mean} Let $P$ be a
random variable following a Poisson distribution with mean $\mu$.
Then for $\epsilon
> 0$
\begin{align*}
&\Pr[P \leq \mu(1-\epsilon)] \leq e^{-\epsilon^2 \mu/2}, \\
&\Pr[P \geq \mu(1+\epsilon)] \leq \left[ e^\epsilon (1+\epsilon)^{-(1+\epsilon)}\right]^\mu < e^{-\frac{1}{2}\epsilon^2 \mu + \frac{1}{2}\epsilon^3\mu}.
\end{align*}
\end{theorem}

\begin{proof}[Proof of Lemma \ref{lem: 2. G(n,p) is SOmim w.r.t. D(m) and
N(u,D(m))}]
By \cite[Theorem 1.2]{B:bollobas01} one has
\begin{align*}
\Pr [\Bin(n,p) = m] &< \extracomputation{\frac{1}{\sqrt{2\pi
p(1-p)n}} \cdot \mathrm{exp}\left\{ -\frac{h^2}{2p(1-p)n}
+\frac{h}{(1-p)n} + \frac{h^3}{p^2 n^2} \right\} = \\
&= \frac{1}{\sqrt{2\pi p(1-p)n}} \cdot \mathrm{exp}\left\{
-\frac{1}{2} \ln n + \ln \ln \ln n + \alpha + O(n^{-1/4}) \right\}
\leq \\ &\leq \frac{e^\alpha}{\sqrt{2\pi p(1-p)n}} \frac{\ln \ln
n}{n} e^{O(n^{-1/4})} \leq } C_1 \frac{\ln \ln n}{n}
\end{align*}
For some constant $C_1$. On the other hand, by the counterpart
\cite[Theorem 1.5]{B:bollobas01} we have
\begin{align*}
&\Pr [\Bin(n,p) = m] > \extracomputation{ \\
&>\frac{1}{\sqrt{2\pi p(1-p)n}} \cdot \mathrm{exp}\left\{
-\frac{h^2}{2p(1-p)n} -\frac{h^3}{2(1-p)^2 n^2} - \frac{h^4}{3p^3
n^3} - \frac{h}{2pn} - \frac{1}{12m} - \frac{1}{12(n-m)} \right\} =
\\
&= \frac{1}{\sqrt{2\pi p(1-p)n}} \cdot \mathrm{exp}\left\{
-\frac{1}{2} \ln n + \ln \ln \ln n + \alpha - O(n^{-1/4}) \right\}
\geq \\
&\geq \frac{e^\alpha}{\sqrt{2\pi p(1-p)n}} \frac{\ln \ln n}{n}
e^{-O(n^{-1/4})} \geq } C_2 \frac{\ln \ln n}{n} , \\
\end{align*}
where $C_2$ is some positive constant. Hence we may
write
\begin{equation} \label{eq: pr[bin(n-1,p)=m] = Theta(log log n / n)}
\Pr[\Bin(n-1,p) = m] = \Theta(\ln \ln n / n) .
\end{equation}
We now argue that if we define the expectation $\mu =
n\Pr[\Bin(n-1,p) = m]$ and set $\epsilon = \ln \mu / \sqrt{\mu}$,
then for any constant $K$ one has
\begin{equation} \label{eq: |D(m)| = mu up to e, whp}
\Pr[(1-\epsilon)\mu \leq |D(m)| \leq (1+\epsilon)\mu] \geq 1 -
O((\log \log n)^{-K}) .
\end{equation}
Indeed, notice that $m$ satisfies the conditions of Theorem
\ref{thm: Dist of |{v|d(v) = m}|} above. Therefore, the total
variation distance between the distribution of $|D(m)|$ and a
Poisson distribution with mean $\mu$ is bounded from above by
$(Cm(m-np)/np) (\log\log n / n) = O((\sqrt{\ln n}\ln\ln
n)/\sqrt{n})$.

By Theorem \ref{thm: Pois is concentrated around its mean} and the
fact that $\mu = \Theta(\ln \ln n))$ we have
\begin{align*}
&\Pr [||D(m)|-\mu| \geq \epsilon\mu] \leq \Pr[|\Pois_\mu - \mu| \geq
\epsilon\mu] + O\left( \frac{\sqrt{\ln n}\ln\ln n}{\sqrt{n}} \right) \leq \\ %
\leq & \left( e^{-\epsilon^2\mu/2} + e^{-\epsilon^2\mu/2+\epsilon^3\mu/2}\right)(1+o(1)) \leq \\ %
\leq & 3e^{-\frac{1}{2} \ln^2\mu} \leq 3\mu^{-\frac{1}{2} \ln \mu} =
O\left(\left(\frac{1}{\log \log n}\right)^K\right)
\end{align*}
for any constant $K$, as desired.

Let $v$ be a vertex of degree $m$. By Lemma \ref{lem: encoding
monadic 2nd order}, a.a.s.\ $G(n,p)$ allows $\varphi(x,v)$-powerset
representation. By Lemma \ref{lem: deg and coden in D(v) as
expected}, $G(n,p)[D(v)]$ is \emph{pseudorandom} in the sense that
the inner degrees and codegress behave as expected from a random
graphs with edge density $p$, that is, the degrees are close (up to
$\epsilon$) to $p|D(m)|$ and the codegrees are close to $p^2
|D(m)|$. Finally, the deterministic Lemma \ref{lem: pseudorandom =>
monadic small set} tells us that every such pseudorandom graph has a
fairly large subset $B$ (logarithmic in the size of the graph) with
the property that each of the subsets of $B$ is the neighborhood in
$B$ of a vertex in the graph.

Let us recapitulate. We have seen that if $v$ has degree $m$ as
above then
\begin{enumerate}
\item a.a.s.\ $|D(v)| = \Theta(\log\log n)$;
\item a.a.s.\ $G(n,p)$ allows $\varphi_1(x,v)$ powerset representation (Lemma \ref{lem: encoding monadic 2nd
order}) and
\item a.a.s.\ there is a vertex $u\in V[G(n,p)]$ such that $B =
N(u,D(v))$ is of logarithmic size, and for every subset $A$ of $B$
there is a vertex $w \in D(v)$ such that $N(w,D(v)) = A$ (Lemma
\ref{lem: deg and coden in D(v) as expected} and Lemma \ref{lem:
pseudorandom => monadic small set}).
\end{enumerate}
Therefore, a.a.s.\ $G(n,p)$ allows a $(\varphi_0(x,v),
\varphi_1(x,u), f)$-double powerset representation with $f =
\Theta(\log \log \log n)$. This completes the proof of the lemma.
\end{proof}

The next fact is straightforward and easy: using second order logic
(over sets) we can interpret arithmetic. While in terms of
expressive power it is enough to have addition and multiplication,
we define a few more relations for the benefit of readability.
\begin{fact} \label{fact: 3. arith in 2nd order}
There exists a second order sentence
\[ \predfont{Arith} \defeq \exists <^2, x_0, \dotsc, x_{100},+^3,
\times^3, P, T, M.\, \varphi(<, x_0, \dotsc, x_{100}, +, \times, P,
T,
M) , \] %
such that if a set $B$ models \predfont{Arith} then
\begin{enumerate}
\item $<$ is a linear order over the elements of $B$;
\item $x_0$ is the minimal element in the order, $x_1$ the second
smallest element and so on up to $x_{100}$;
\item $+,\times,P,T$ and $M$ act as the standard addition,
multiplication, base two power, base two tower\footnote{The
\emph{base 2 tower} function $t(n)$ is defined by $t(0)=1$ and $t(n)
= 2^{t(n-1)}$ for all $n\geq 1$.} and modulo 100 relations.
\end{enumerate}
In addition, for every set $B$ of cardinality at least 101, one has
$B \models \predfont{Arith}$.
\end{fact}
Proving Fact \ref{fact: 3. arith in 2nd order} is fairly standard
and not difficult. The proof may be found, e.g., in \cite[Chapter
8]{B:spencer01}.

Having arithmetic as above, that is, if $B\models\predfont{Arith}$,
we can express sentences dealing with the size of the set $B$. We
shall use the following sentence saying that $0 \leq \logstar |B| <
50 \imod{100}$.
\begin{align} \label{eq: defining LogStar}
\predfont{LogStar} \defeq \exists x \in B.\, &[(\exists y \in B.\, T(x,y)) \wedge \\ %
& (\forall z \in B.\, ((x<z) \rightarrow (\neg\exists y\in B.\, T(z,y)))) \wedge \nonumber \\%
& M(x,0) \vee M(x,1) \dotsb \vee M(x,49) ] . \nonumber
\end{align}

Now we have all the needed ingredients to prove Theorem \ref{thm:
L(Q_ham) has arithmetic}.
\begin{proof}[{proof of Theorem \ref{thm: L(Q_ham) has arithmetic}}]
\label{proof: L(Q_ham) has arithmetic} Assume $p \leq 1/2$ (if
$p>1/2$ replace any $\sim$ by $\not \sim$). We begin by showing that
given $v$, the set $D(v)$ is $\Logic(Q_\Ham)$-definable.

Let $A$ and $B$ be two $\Logic(Q_\Ham)$-definable sets (in our
application both will usually be neighborhoods of vertices which
are, of course, first order definable, and hence definable in any
semiregular language). We say that $|A| \preceq |B|$ if there is a
Hamilton cycle in the graph over $[n]$ having all the edges between
$A\setminus B$ and $B\setminus A$ and all the edges with both
endpoints in $V\setminus(A \setminus B)$. Clearly, $|A| \preceq |B|$
is expressible in $\Logic(Q_\Ham)$, provided that $A$ and $B$ are
$\Logic(Q_\Ham)$-definable. Observe that per definition $|A|
\preceq |B|$ if and only if $|A| \leq |B|$.
\comment{ %
\begin{align*}
|A| \preceq |B| \defeq Q_\Ham xy.\, %
[&(x \in A \setminus B \wedge y \in B \setminus A) \vee (y \in A
\setminus B \wedge x \in B \setminus A) \vee \\ %
&(x \in B \setminus A \wedge y \notin A \triangle B) \vee (y \in B
\setminus A \wedge x \notin A \triangle B) \vee \\ %
&(x \notin A \triangle B \wedge y \notin A \triangle B)] .
\end{align*}
That is, we partition the vertices of the graph into three sets,
$A\setminus B, B\setminus A$ and $V \setminus (A \triangle B)$, add
all the edges between $A\setminus B$ and $B\setminus A$, all the
edges between $B\setminus A$ and $V \setminus (A \triangle B)$ and
all the edges inside $V \setminus (A \triangle B)$. Then we check if
the resulting graph is Hamiltonian. Observe that this graph is
Hamiltonian if and only if $|A| < |B|$ and $|V \setminus (A
\triangle B)| + |A\setminus B| \geq |B \setminus A|$ or if $|A| =
|B|$ and $V \setminus (A \triangle B) = \varnothing$.

Since $p\leq 1/2$ we have that a.a.s.\ for every pair of vertices
$u,v$ one has $|V \setminus (N(u) \triangle N(v))| = \frac{1}{2} n +
\tilde{\Omega}(\sqrt{n})$ while $|N(u)\setminus N(v)| = \frac{1}{4}n
+ \tilde{O}(\sqrt{n})$ and $|N(v)\setminus N(u)| = \frac{1}{4}n +
\tilde{O}(\sqrt{n})$. Therefore a.a.s.\ for any two vertices $u,v$
we have $|N(u)| \preceq |N(v)| \Leftrightarrow |N(u)| < |N(v)|$. We
shall use $\preceq$ only over neighborhoods of vertices or over
subsets of neighborhoods. Hence in our application a.a.s.\ $|A|
\preceq |B|$ if and only if $|A| < |B|$.
} %

The last two paragraphs say that membership in $D(v)$, and hence
also $\varphi_0(x,v)$, is expressible in $\Logic(Q_\Ham)$ since we
may simply express $x \in D(v)$ or $\varphi_0(x,v)$ by writing
$(|N(x)| \preceq |N(v)|) \wedge (|N(v)| \preceq |N(x)|)$.

At this point we have everything needed in order to apply Lemma
\ref{lem: 2. G(n,p) is SOmim w.r.t. D(m) and N(u,D(m))} to encode
any second order formula over a set of size $\log \log \log (n)$. We
shall demonstrate the process for a specific formula having no
limiting probability (also in the modular sense).

By Fact \ref{fact: 3. arith in 2nd order} there is a second order
formula, $\predfont{Arith}$, that holds for every graph (actually,
every set) with size at least 101 and lets us express arithmetic.

Consider the set $D(v)$ for some vertex $v$ and a vertex $u \in V$.
Using the encoding of Lemma \ref{lem: 1. G SOmim -> encode 2nd
order} with the formulas $\varphi_0(v)$ and $\varphi_1(u)$, we
encode the sentence $\predfont{Arith}$ aforementioned into
$\predfont{Arith}'(u,v)$. Similarly we encode the sentence
$\predfont{LogStar}$ from Equation (\ref{eq: defining LogStar}) into
a sentence $\predfont{LogStar}'(u,v)$. By Lemma \ref{lem: 1. G SOmim
-> encode 2nd order} we have
\[ N(u,D(m)) \models \predfont{Arith} \Longleftrightarrow G(n,p) \models \predfont{Arith}'(u,v) , \]
and similarly for $\predfont{LogStar}$ and
$\predfont{LogStar}'(u,v)$.

Let $m$ be as in Equation \ref{eq: definition of m and h} and let
$v$ be a vertex of degree $m$. By Lemma \ref{lem: 2. G(n,p) is SOmim
w.r.t. D(m) and N(u,D(m))} there exists a vertex $u$ such that
$G(n,p)$ allows a $\varphi_0, \varphi_1,f$-double powerset
representation with $f = \Theta(\log \log \log n)$. Therefore we
have that a.a.s.\ $G(n,p) \models \predfont{Arith}'(u,v)$ for some
vertices $u$ and $v$. Now we may write the nonconverging sentence.
\begin{align*}
\predfont{Non}&\predfont{Conv} \defeq \exists v,u.\, [\predfont{Arith}'(u,v) \wedge \\ %
& (\forall v',u'.\, (\predfont{Arith}'(u,v)
\rightarrow \neg(|N(u',D(v'))| \preceq |N(u,D(v))|))) \wedge \\
& \predfont{LogStar}'(u,v) ] .
\end{align*}
\predfont{NonConv} says that if we consider the maximal integer $b$
for which there is a pair of vertices $v,u$ such that $d(u,D(v)) =
b$ and $\predfont{Arith}'(v,u)$ holds, then $0 \leq \logstar(b) < 50
\imod{100}$. By the above, we know that this maximal $b$ satisfies $
c \log \log \log n \leq b$ for some constant $c$. We do not bother
with the upper bound, rather we simply mention that trivially $b
\leq n$. This means that $\logstar(n)-4 \leq \logstar(b) \leq
\logstar(n)$.

Thus, if we consider an infinite sequence of numbers $(n_i)$ all
having $\logstar(n_i) = 49 \imod{100}$, then
$\lim_{i\rightarrow\infty} \Pr[G(n_i,p) \models \predfont{NonConv}]
= 1$. On the other hand taking another sequence $(n'_i)$ such that
$\logstar(n'_i) = 99 \imod{100}$ gives $\lim_{i\rightarrow\infty}
\Pr[G(n'_i,p) \models \predfont{NonConv}] = 0$. Clearly the limiting
probability of $\predfont{NonConv}$ violates the modular convergence
law of \cite{IP:kolaitis_&_kopparty2009} as well.

A similar encoding may be applied for any second order formula. The
proof is complete.
\end{proof}

\subsection{Three technical lemmas}
\begin{lemma} \label{lem: encoding monadic 2nd order}
Let $0 < p < 1$ be constant and consider $G(n,p)$. Then a.a.s\ for
every set of vertices $S$ of size $|S| \leq C \log \log n$ (for some
constant $C>0$) the following holds: For every subset $A \subset S$
there is a vertex $v \notin S$ such that
\[ N(v) \cap S = A . \]
\end{lemma}
\begin{proof}
Let $A \subset S$ and let $v \notin S$ be a vertex. The probability
that $N(v) \cap S = A$ is exactly $p^{|A|}(1-p)^{\binom{|S|}{2} -
|A|} \geq 2^{-c(C \log \log n)^2}$ where $c = \log (\min(p,1-p))$ is
constant depending only on $p$. Therefore the probability that there
is no witness for the set is bounded by $(1 - 2^{-c'(\log \log
n)^2})^{n - C \log \log n} \leq e^{-\sqrt{n}}$. Apply a union bound
over all the $2^{C \log \log n} \leq (\log n)^C$ possible subsets of
$S$ and then another union bound over all the $\sum_{k=1}^{C \log
\log n} \binom{n}{k} \leq \binom{n}{C \log \log n + 1} \leq e^{C'
\log n \log \log n}$ sets of size $k \leq C \log \log n$. Hence we
have that a.a.s.\ for every set $S$ of size at most $C \log \log n$
there is a witness for every subset, and the proof is complete.
\end{proof}

Before stating and proving the nest lemma we cite yet another
Chernoff type bound, this time regarding the tails of the
hypergeometric distribution.
\begin{theorem}[{\cite[Theorem 2.10]{B:janson_&_luczak_&_rucinski00}}]
\label{thm: hypergeom is concentrated around its mean}
Let $X$ be a random variable following a hypergeometric distribution
with parameters $n, l$ and $m$. Let $\mu = \E X = lm/n$ and $t \geq
0$. Then
\begin{align*}
&\Pr[X \geq \E X + t] \leq e^{-t^2 / (2(\mu + t/3))} ; \\
&\Pr[X \leq \E X - t] \leq e^{-t^2 / (2\mu)} .
\end{align*}
\end{theorem}

Now we are ready to prove the following lemma, basically saying that
degrees and codegrees inside $D(m)$ behave as expected.

\begin{lemma} \label{lem: deg and coden in D(v) as expected}
Let $0<p<1$ be constant and let $m$ be as above. Set $\epsilon =
\ln\ln\ln n/ \sqrt{\ln\ln n}$. Then a.a.s.\ for every $v \in D(m)$
one has
\[ d(v,D(m)) = (1 \pm \epsilon) p|D(m)| , \]
and a.a.s.\ for every pair of vertices $u,v$ one has
\[ |N(u) \cap N(v) \cap D(m)| = (1 \pm \epsilon) p^2|D(m)| . \]
\end{lemma}
\begin{proof}
\comment{Let $m=m(n)>np$ be chosen so that
$\Pr[\Bin(n-1,p)=m]=(1+o(1))\log\log n/n$ (for example, taking $m$
as in Lemma \ref{lem: exists v with |D(v)| = ln ln n} will work).}

Denote $S=\left\{v \mid d(v)=m \right\} = D(m)$. We are interested
in typical degrees and co-degrees inside $G[S]$.

Choose $v\in [n]$. Denote $G_v=G[V\setminus \{v\}]$. Let $S_0=\{u
\in G_v \mid d_{G_v}(u)=m-1\}$ and $S_1=\{u\in G_v \mid
d_{G_v}(u)=m\}$. Then, by definition, $d(v,S)=d(v,S_0)$. Denote
$X_0=|S_0|, E[X_0]=\mu_0, X_1=|S_1|, E[X_1]=\mu_1$. Notice that
$\mu_0$ and $\mu_1$ are very close to each other since
\begin{align*}
\frac{\mu_0}{\mu_1} = %
\frac{(n-1)\binom{n-2}{m-1} p^{m-1} (1-p)^{n-2-(m-1)}}
{(n-1)\binom{n-2}{m-2} p^{m-2} (1-p)^{n-2-(m-2)}} =
\frac{(n-m)p}{(m-1)(1-p)} = 1 \pm O(1/\sqrt{n}) .
\end{align*}

Hereinafter equations containing $X_i, \mu_i$ or $S_i$ are to be
taken as two equations, one for $i=0$ and one for $i=1$.

By Equation (\ref{eq: pr[bin(n-1,p)=m] = Theta(log log n / n)}) we
have $\mu_i=\Theta(\log\log n)$. By Equation (\ref{eq: |D(m)| = mu up to e, whp}) we know that $X_i$ is concentrated around its mean
$\mu_i$, that is, $\Pr[|X_i-\mu_i|>\epsilon \mu_i]< (\log\log
n)^{-K}$ for any constant $K>0$.

Now, $|S|=d(v,S_0)+(|S_1|-d(v,S_1))$ and, as already mentioned,
$d(v,S)=d(v,S_0)$.

Consider the event $A$ being ``$v\in S \wedge |d(v,S)-p|S||>\epsilon
|S|$''. First expose the degree of $v$, $\Pr[d(v)=m]=O(\log\log
n/n)$. Condition now on $d(v)=m$, and expose $G_v$, with probability
$1-O((\log\log n)^{-K})$, the random variables $X_0$ and $X_1$ are
close to their expectations and in particular are nearly equal.

Now, given $d(v)$ and the sets $S_0$, $S_1$, the random variables
$d(v,S_0)$ and $d(v,S_1)$ are distributed hypergeometrically with
parameters $n-1, X_i, m$. Let $\lambda_i = \E [d(v,S_i)] = m X_i /
(n-1)$, and notice that $\lambda_i = (np+h)X_i/(n-1) = pX_i(1 \pm
\epsilon) = p\mu_i(1\pm 2\epsilon)$. Applying Theorem \ref{thm:
hypergeom is concentrated around its mean} we get
\begin{align*}
& \Pr[|d(v,S_i) - \lambda_i| \geq \epsilon \lambda_i] \leq
e^{-(\epsilon \lambda_i)^2/(2\lambda_i(1+\epsilon/3))} +
e^{-(\epsilon \lambda_i)^2/(2\lambda_i)} \leq \\
& \leq 3 e^{-(\epsilon^2 \lambda_i)/2} = O\left(\frac{1}{\log \log
n}\right)^K ,
\end{align*}
for every positive constant $K$.

All in all, the probability of the event $A$ comes out to be
$O(\log\log n/n) \cdot O((\log\log n)^{-K})=o(1/n)$, and therefore a
union bound over the vertex set is applicable and the first part of
the lemma is proven.

We repeat the argument for codegrees. Fix two vertices $u,v$ and
consider $G_{u,v} = G \setminus \{u,v\}$. Let $A$ be the event
``$u,v\in S \wedge |\codeg(u,v,S)-p^2|S||>\epsilon |S|$''. We first
expose the edge between $u$ and $v$ and assume w.l.o.g\ that $u$ and
$v$ are not adjacent. Next we expose the degrees of $u$ and $v$ and
with probability $\Theta((\log \log n/ n)^2)$ we have that $d(u) =
d(v) = m$ (these events being independent). Next we expose the edges
of $G_{u,v}$ and look ar the sets $S_0, S_1$ and $S_2$ defined
similarly to the above --- $S_i = \{w \in G_{u,v} \mid
d_{G_{u,v}}(w) = m-2+i\}$. This time we have
\[ S = (N(u) \cap N(v) \cap S_0) \cup ((N(u) \triangle N(v)) \cap
S_1 ) \cup (S_2 \setminus (N(u) \cup N(v))) ,
\]
and thus
\begin{align*}
|S| = &\codeg(u,v,S_0) + (d(u,S_1) + d(v,S_1) - \codeg(u,v,S_1)) +
\\
&(|S_2| - d(u,S_2) - d(v,S_2) + \codeg(u,v,S_2)) .
\end{align*}
Again, we know that all the summands in the right hand side of the
equation above are close to their expectation with high enough
probability. That is, by Equation (\ref{eq: |D(m)| = mu up to e,
whp}), with probability $(1-O((\log \log n)^{-K}))$ the sets $S_0,
S_1$ and $S_2$ are all of size $\Theta(\log \log n)$ and are all
nearly of the same cardinality (the difference being of order
$\epsilon \log \log n$). The degrees $d(u,S_i)$ and $d(v,S_i)$
follow a hypergeometric distribution with parameters $n-2, |S_i|$
and $m$, and with probability $(1-O((\log \log n)^{-K})$ they are
all concentrated around their means. Finally notice that the
codegree of $u$ and $v$ in $S_i$ is determined by the degree of $u$
inside $N(v) \cap S_i$, and thus it also follows a hypergeometric
distribution with parameters $n-2, p|S_i|(1\pm\epsilon)$ and $m$,
and the same hypergeometric tails bound applies. All in all we get
that the probability of $A$ is bounded by
\[ O\left( \frac{\log \log n}{n^2} \right) \cdot O\left(
\frac{1}{(\log \log n)^K}\right) = o(n^{-2}).
\]
Applying the union bound over all pairs of vertices completes the
proof.
\end{proof}

The next lemma deals with pseudorandom graphs, saying that if the
degrees and codegrees of a graph are similar to those of a random
graphs, then there is a small set of vertices such that every subset
of it is induced by a vertex of the graph.

\begin{lemma} \label{lem: pseudorandom => monadic small set}
Let $0<p<1$ be constant. Let $G$ be a graph on $n$ vertices and let
$\epsilon = n^{-b}$ where $0<b<1$ is some constant. Assume that $G$
satisfies $\delta(G) \geq (p-\epsilon) n$ and $\Delta(G) \leq
(p+\epsilon) n$. Additionally assume that for every two vertices
$u,v$ one has $(p^2-\epsilon)n \leq \codeg(u,v) \leq (p^2+\epsilon)n
$. Then there is a constant $c=c(b,p)>0$ and a set of vertices $S$
of size $s = c \log n$ such that for every subset $A \subset S$
there is a vertex $v \in G$ having $N(v,S) = A$.
\end{lemma}
\begin{proof}
Fix $0 \leq a \leq s$. Pick a set $A$ of size $a$ uniformly at
random. Let $X_A$ be the random variable counting the number of
vertices $v$ in $G$ satisfying $A \subset N(v)$. We claim:\newline
{\bf Claim \ref{lem: pseudorandom => monadic small set}.1:} $\E
[X_A] = np^a(1+\tilde{O}(n^{-b}))$ and $\VAR[X_A] =
\tilde{O}(np^an^{1-b})$.
\newline {\bf Proof of Claim \ref{lem: pseudorandom =>
monadic small set}.1:} For every $v\in G$ let $X_{A,v}$ be the
indicator random variable for the event $A \subset N(v)$. Let $d =
d(v)$ be the degree of $v$. Then,
\[ \Pr[A \subset N(v)] = \frac{\binom{d}{a}}{\binom{n}{a}} =
\frac{d(d-1)\dots(d-a+1)}{n(n-1)\dots(n-a+1)} .
\]
Since ($a = O(\log n)$ and)
\[ \frac{d-a}{n-a} = \frac{d}{n} \left(1 + \frac{a}{n-a}
\left(1-\frac{n}{d}\right)\right) = \frac{d}{n} \left(
1+O\left(\frac{a}{n}\right) \right),
\]
we have
\[ \Pr[A \subset N(v)] = \left( \frac{d}{n} \right)^a
\left(1+O\left(\frac{a^2}{n}\right) \right) = p^a(1\pm \epsilon)^a
\left(1+O\left(\frac{a^2}{n}\right) \right) .
\]
Therefore,
\begin{equation} \label{eq: E x_A in lem: pseudorandom => monadic small
set}
\E X_A = \sum X_{A,v} = np^a (1\pm \tilde{O}(n^{-b})) .
\end{equation}

Next we wish to estimate $\VAR [X]$. Consider two vertices $u,v$.
\[ \E[X_{A,u} X_{A,v}] = \Pr[A \subset N(u) \wedge A \subset N(v)] , \]
which is the probability of $A$ being a subset of the common
neighborhood of $u$ and $v$. Let $d(u,v)$ be the codegree of $u$ and
$v$. Since $d(u,v)$ is bounded we get (similarly to the above):
\[ \Pr[A \subset N^*(u,v)] = \frac{\binom{d(u,v)}{a}}{\binom{n}{a}}
= \left( \frac{d(u,v)}{n} \right)^a
\left(1+O\left(\frac{a^2}{n}\right) \right) ,
\]
which yields
\[ \E[X_{A,u} X_{A,v}] = p^{2a}(1\pm O(\epsilon a)) = p^{2a}(1\pm O(an^{-b})) .
\]
Now,
\begin{align*}
&\COV[X_{A,u}, X_{A,v}] = \E[X_{A,u} X_{A,v}] - \E[X_{A,u}]
\E[X_{A,u}] \leq \\
&p^{2a}(1\pm O(an^{-b})) - (p^a (1\pm O(an^{-b})))^2 = O(p^{2a} a
n^{-b}) .
\end{align*}
Thus we have
\begin{align*}
&\VAR[X_A] \leq \E[X_a] + \sum_{u \neq v} \COV[{A,u}, X_{A,v}] =
np^a(1+o(1)) + O(n^{2-b}p^{2a} a ) ,
\end{align*}
which gives the desired expression
\[ \VAR[X_A] = \tilde{O}(np^an^{1-b})). \]
\newline{\bf End of proof of Claim \ref{lem: pseudorandom => monadic small set}.1}

Knowing $\E [X_A]$ and $\VAR [X_A]$ we can apply Chebyshev's
inequality. Let $\delta = n^{3c \log(1-p)}$ and recall that $0<a<s =
c\log n$. We may write
\begin{align*}
\Pr [|X_A - \E [X_A]| \geq & 3\delta \E [X_A]] \leq \frac{ \VAR
[X_A]}{9\delta^2 \E[X_A]^2} =
O\left( \frac{n^{-6c \log(1-p)}np^an^{1-b}} {n^2p^{2a} (1+o(1))} \right) = \\
= & O(n^{-6c \log(1-p) - b -c\log p}) = O(n^{-c'}) ,
\end{align*}
where $c' = b+c\log p+6c \log(1-p)$.

Recall that by Equation (\ref{eq: E x_A in lem: pseudorandom =>
monadic small set}) we have $\E [X_A] = np^a(1 \pm
\tilde{O}(n^{-b}))$. If we require $c < b/(-3 \log(1-p))$ we have
$\delta = \tilde\omega(n^{-b})$ and thus
\begin{align*}
& \Pr[|X_A - np^a| \geq \delta np^a] \leq \\
& \Pr[|X_A - \E X_A| \geq (\delta + \tilde{O}(n^{-b})) np^a] \leq \\
& \Pr[|X_A - \E X_A| \geq 2\delta np^a] \geq \\
& \Pr[|X_A - \E X_A| \geq 3\delta \E[X_A]] = O(n^{-c'}).
\end{align*}

Call a set $A$ having $|X_A - np^a| \geq \delta np^a$ ``bad''. Every
set of size $a$ is a subset of $\binom{n}{s-a}$ sets of size $s$.
Hence the number of $s$-sets containing a bad set is bounded by
\begin{align*}
&\sum_{a=0}^s \binom{n}{a} O(n^{-c'}) \binom{n}{s-a} = %
O(n^{-c'}) \sum_{a=0}^s \binom{n}{a} \binom{n}{s-a} = \\
= & O(n^{-c'}) \sum_{a=0}^s \binom{n}{s} \binom{s}{a} = %
O(n^{-c'}) \binom{n}{s} \sum_{a=0}^s \binom{s}{a} = \\
& = O(n^{-c'}) \binom{n}{s} 2^s = O(n^{-c'+c}) \binom{n}{s}.
\end{align*}
Hence, requiring $c-c' < 0$ or $c < b/(1-\log p - 6\log (1-p))$
gives that the number of $s$-sets having a bad subset is
$o(\binom{n}{s})$.

Pick a set $S$ of size $s$ without any bad subset. Consider $A
\subset S$, $|A| = a$. Let $W(A)$ be the set of vertices $v$ such
that $A \subset N(v)$ and let $W^*(A)$ be the set of vertices $v$
such that $A = N(v)$. By the inclusion exclusion principle we have
\begin{align*}
&|W^*(A)| = |W(A)| - \sum_{v \in S \setminus A} |W(A \cup {v})| +
\sum_{u\neq v \in S \setminus A} |W(A \cup {u,v})| - \dotsb = \\
= & np^a \sum_{k=0}^{s-a} (-1)^k \binom{s-a}{k} p^k (1\pm \delta) \geq \\
\geq & np^a \sum_{k=0}^{s-a} (-1)^k \binom{s-a}{k} p^k  - \delta
np^a \sum_{k=0}^{s-a} \binom{s-a}{k} p^k = \langle \text{ IE
principle\footnotemark\ } \rangle = \\
= & np^a (1-p)^{s-a}  - \delta np^a \sum_{k=0}^{s-a} \binom{s-a}{k} p^k = \\
= & np^a (1-p)^{s-a}  - \delta np^a \sum_{k=0}^{s-a} \frac{\binom{s-a}{k} p^k (1-p)^{s-a-k}}{(1-p)^{s-a-k}} \geq \\
= & np^a ((1-p)^{s-a} - \delta/(1-p)^{s-a}) \geq np^a ((1-p)^s -
\delta/(1-p)^s).
\end{align*}
\footnotetext{Assume you have $s-a$ biased coins that yield head
with probability $p$. Let $B$ be the event ``No head when tossing
all coins'' (we assume independence of course). Define $A_i$ as the
event ``The $i$'th coin gave head''. Now, by the inclusion exclusion
principle the probability of $B$ is $\sum_{k=0}^{s-a} (-1)^k
\binom{s-a}{k} p^k$. Direct computation gives that $\Pr[B] =
(1-p)^{s-a}$.}

Since $s=c\log n$ and $\delta = n^{3c\log(1-p)} = (1-p)^{3s}$ we
have
\begin{align*}
|W^*(A)| \geq & np^a(1-p)^s (1+o(1)) \geq np^s(1-p)^s (1+o(1)) \geq
\\
\geq & n^{1 + c\log p + c\log(1-p)} .
\end{align*}
In particular, if we require $c < -1/(\log p + \log(1-p))$, we get
that $W^*(A)$ is not empty and the proof is complete.
\end{proof}

\noindent {\bf  Acknowledgment} The first author thanks Michael
Krivelevich for many helpful and instructive discussions regarding
these problems in general, and the technical aspects of the proofs
in particular.

\noindent Thanks also to Joel Spencer for introducing the problems
and the beauty of this subject.

\bibliographystyle{abbrv}

\end{document}